\documentclass[11pt]{amsart}

\usepackage[english]{babel}
\usepackage[ansinew]{inputenc}
\usepackage{color}
\usepackage{amsmath, amsthm, amssymb, graphicx, mathrsfs}
\usepackage{hyperref}
\usepackage{pinlabel}
\usepackage{tikz}

\usepackage[headings]{fullpage}

\usepackage{amsmath,amsthm,amssymb}

\newtheorem{thm}{Theorem}[section]
\newtheorem{lem}[thm]{Lemma}

\theoremstyle{definition}

\newtheorem{prop-defn}[thm]{Proposition-Definition}

\newtheorem*{subcase*}{Subcase}

\newtheorem*{claim*}{Claim}
\newtheorem*{ack*}{Acknowledgements}
\newtheorem*{ex*}{Example}
\newtheorem*{case*}{Case}

\newtheoremstyle{named}{}{}{\itshape}{}{\bfseries}{.}{.5em}{#3}
\theoremstyle{named}
\newtheorem*{namedtheorem}{Theorem}

\newcommand{\RR}{\mathbb{R}}
\DeclareMathOperator{\re}{Re}
\DeclareMathOperator{\im}{Im}
\DeclareMathOperator{\slope}{slope}
\DeclareMathOperator{\Li}{Li}

\definecolor{AH}{HTML}{FF0000}
\definecolor{O2}{HTML}{8080FF}
\definecolor{BD}{HTML}{4C8F8A}
\definecolor{HF}{HTML}{FFD300}
\definecolor{AE}{HTML}{B266B2}

\title{The distribution of gaps for saddle connections on the octagon}
\author{Caglar Uyanik and Grace Work}
\address{\tt Department of Mathematics, University of Illinois at
 Urbana-Champaign, 1409 West Green Street, Urbana, IL 61801, USA
\newline \href{https://sites.google.com/site/caglaruyanik/}{https://sites.google.com/site/caglaruyanik/}
\newline \href{http://www.math.uiuc.edu/~work2}{http://www.math.uiuc.edu/\~{}work2}} \email{\tt cuyanik2@illinois.edu}
\email{\tt work2@illinois.edu}
\thanks{\today}

\begin{document}

\begin{abstract} We explicitly compute the limiting gap distribution for slopes of saddle connections on the flat surface associated to the regular octagon with opposite sides identified. This is the first such computation where the Veech group of the translation surface has multiple cusps. We also show how to parametrize a Poincar\'e section for the horocycle flow on $SL(2,\RR)/SL(X,\omega)$ associated to an arbitrary Veech surface $(X, \omega)$. As a corollary, we show that the associated gap distribution is piecewise real analytic.

\end{abstract}

\thanks{Both authors acknowledge support from U.S. National Science Foundation grants DMS 1107452, 1107263, 1107367 ``RNMS: GEometric structures And Representation varieties" (the GEAR Network). The first author was partially supported by the NSF grants of Ilya Kapovich (DMS-1405146) and Christopher J. Leininger (DMS-1207183). The second author was partially supported by the NSF grant of Jayadev Athreya (DMS-1351853).
}

\maketitle


\section{Introduction}

\subsection{Gap distribution for the Octagon}

Consider the surface $\mathcal{O}$ obtained by identifying parallel sides of the regular octagon by Euclidean translations, as shown in Figure \ref{octagon}. This results in a genus 2 translation surface with a single cone point of cone angle $6\pi$. A \emph{saddle connection} on this surface is a straight line connecting the cone point to itself. To each saddle connection, $\gamma$, we associate a holonomy vector, $v_\gamma \in \mathbb{C}$,  that records how far it travels in each direction. For example, the saddle connection drawn in Figure \ref{octagon} would have holonomy vector $2(1+\sqrt{2}) + i$.

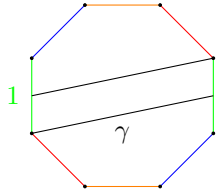
\begin{figure}[h!]
\centering
\begin{tikzpicture}
\coordinate (A) at (0,0);
\coordinate (B) at (-.707,.707);
\coordinate (C) at (1,0);
\coordinate (D) at (-.707,1.707);
\coordinate (E) at (0, 2.414);
\coordinate (F) at (1, 2.414);
\coordinate (G) at (1.707, 1.707);
\coordinate (H) at (1.707, .707);
\coordinate (I) at (1.707,1.207);
\coordinate (J) at (-.707,1.207);
\draw [red] (A) -- (B);
\draw [blue] (D) -- (E);
\draw [orange] (E) -- (F);
\draw [red] (F) -- (G);
\draw [green] (G) -- (H);
\draw [green] (B) -- (D) node[left=.25mm,pos=.5,color=green]{\small $1$};
\draw [blue] (H) -- (C);
\draw [orange] (C) -- (A);
\draw (B) -- (I) node[below=.25mm,pos=.5]{\small $\gamma$};
\draw (J) -- (G);
    \draw[fill] (B) circle (.5pt);
    \draw[fill] (C) circle (.5pt);
    \draw[fill] (D) circle (.5pt);
    \draw[fill] (E) circle (.5pt);
    \draw[fill] (F) circle (.5pt);
        \draw[fill] (A) circle (.5pt);
    \draw[fill] (G) circle (.5pt);
    \draw[fill] (H) circle (.5pt);
\end{tikzpicture}
\caption{The regular octagon with parallel sides identified to obtain the translation surface $\mathcal{O}$. A saddle connection $\gamma$ has been drawn connecting the cone point to itself.}\label{octagon}
\end{figure}



We are interested in the slopes of holonomy vectors coming from the saddle connections on the octagon. Following the convention introduced by \cite{Ath13,ACL}, we refer to them as ``slopes of saddle connections". Since the set of vectors is symmetric with respect to coordinate axes, we restrict our attention to slopes in the first quadrant.
 
Let
\[
\mathbb{S}_{\mathcal{O}}=\left\{\slope(v_{\gamma})=\frac{\im(v_\gamma)}{\re(v_\gamma)}\mid  \text{$\gamma$ is a saddle connection on $\mathcal{O}$ and}\ 0<\re(v_\gamma), 0\le \im(v_\gamma)\right\}
\]
We write $\mathbb{S}_{\mathcal{O}}$ as an increasing union of $\mathbb{S}_{\mathcal{O}}^R$ for $R\to\infty$, where
\[
\mathbb{S}^R_{\mathcal{O}} := \left\{\slope(v_{\gamma})\mid \text{$\gamma$ is a saddle connection on $\mathcal{O}$ and}\ 0<\re(v_\gamma)\le R, 0\le \im(v_\gamma)\le R \right\}. 
\] In other words, $\mathbb{S}_{\mathcal{O}}^R$ is the set of (finite) non-negative slopes of saddle connections on $\mathcal{O}$ that lie in the first quadrant and in the $\ell_\infty$ ball of radius $R$ around the origin. Let $N(R)=|\mathbb{S}^R_{\mathcal{O}}|$ be the cardinality of the set $\mathbb{S}^R_{\mathcal{O}} $. Veech showed that $N(R)$ grows quadratically with $R$, moreover the saddle connection directions on any Veech surface $(X,\omega)$, in particular $\mathcal{O}$, equidistribute on $S^1$ with respect to Lebesgue measure on the circle, see \cite{Vee98}.

Write $\mathbb{S}_{\mathcal{O}}^R$ as a set of increasing slopes:
\[
\mathbb{S}_{\mathcal{O}}^R=\left\{0\le s_0^R < s_1^R < s_2^R < \cdots < s_{N(R)-1}^R\right\}. 
\]
Since $N(R)$ grows quadratically, it is natural to define the set of \emph{renormalized slope gaps} for $\mathcal{O}$ as \[
\mathbb{G}_{\mathcal{O}}^R := \left\{R^{2}(s_{i}^R - s_{i-1}^R) : 1 \leq i \leq N(R)-1, s_i \in \mathbb{S}_{\mathcal{O}}^R\right\}.
\]



Our main theorem describes the asymptotic behavior for the distribution of the set of renormalized slope gaps:

\begin{thm}\label{octthm} There is a limiting probability distribution function $f:[0,\infty)\to[0,1]$ such that 
\[
\lim_{R\to\infty}\frac{|\mathbb{G}_{\mathcal{O}}^{R}\cap(a,b)|}{N(R)}=\int_{a}^{b}f(x)dx.
\]
The function $f$ is continuous, piecewise differentiable with seven points of non-differentiability. Each piece is expressed in terms of elementary functions. The graph of the function $f$ is shown below and an explicit description is given in Section \ref{verticalstrip}.  

\begin{figure}[h!]
\centering
\includegraphics[scale=0.8]{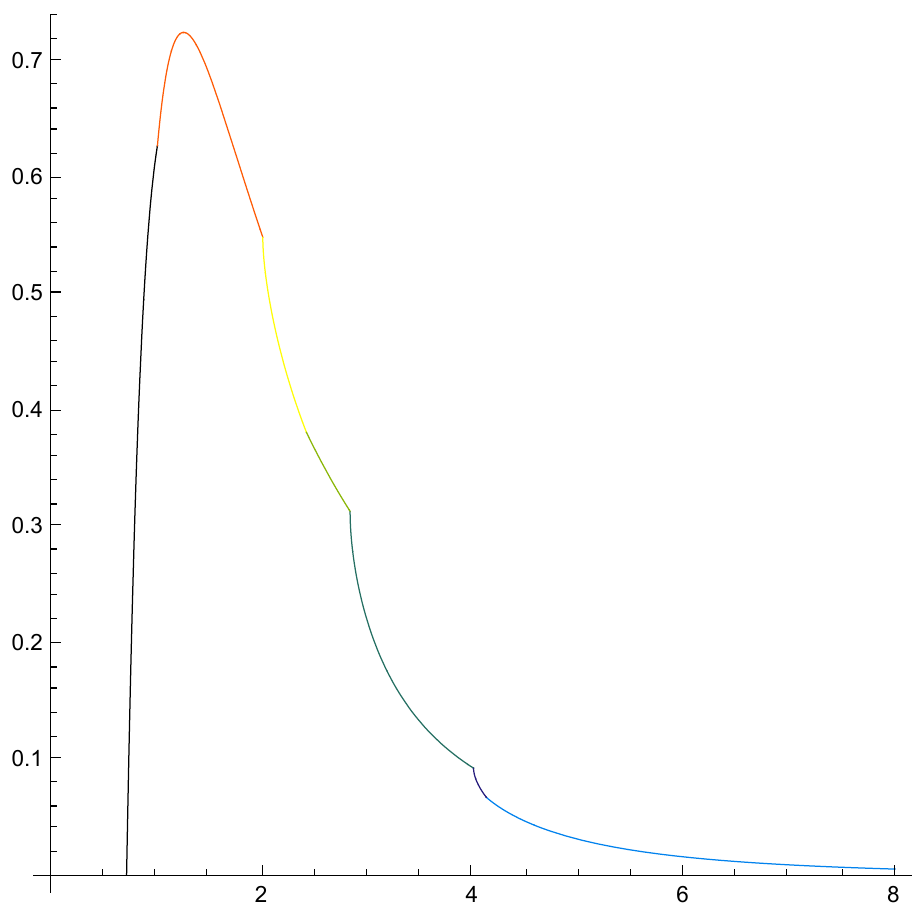}
\caption{Limiting gap distribution for the octagon}
\label{dist}
\end{figure}
\end{thm}

Theorem \ref{octthm} gives information on finer statistics for the set of slopes of saddle connections. In particular, it reveals an underlying structure because we see that the gap distribution is not exponential and has quadratic tail. Were the slopes of saddle connections \emph{``truly random''}, one would expect the gap distribution to be exponential, as is the case with independent identically distributed uniform random variables on $[0,1]$.

\subsection{Translation Surfaces}\label{translationsurfaces}

In this section we review the necessary background related to translation surfaces, which we use in Theorem \ref{parametrizingsection}. We refer reader to \cite{HS,Masur,MT,Zorich} for a detailed treatment of the subject. A \emph{translation surface} is a pair $(X,\omega)$ where $X$ is a closed Riemann surface of genus $g$, and $\omega$ is a holomorphic $1$-form on $X$. Equivalently, a \emph{translation surface} is a finite union of polygons $\{P_1,\dotsc P_k\}$ in the Euclidean plane such that for each edge there exists a parallel edge of the same length and these pairs are glued together by a Euclidean translation. It follows that the total angle around each vertex is equal to $2\pi c_i$ for some integer $c_i\ge1$. We note here that we are considering a fixed embedding of the above polygons into 
$\mathbb{R}^2\cong \mathbb{C}$ up to translations, that is we are considering polygons in the Euclidean plane with a preferred vertical direction. Since translations are holomorphic functions that preserve the standard one form $dz$ on the Euclidean plane, we get a Riemann surface structure on the glued
surface together with a holomorphic $1$-form. The set of vertices where the total angle $\theta>2\pi$
is
called the set of \emph{cone points}, and
each cone point corresponds
to a zero of the holomorphic $1$-form. At a zero of order $n$ the total angle is $2\pi(n+1)$.  

A \emph{saddle connection} is a geodesic connecting two of the cone points without any cone points in the interior. An oriented saddle connection $\gamma$ in $X$ determines a \emph{holonomy vector} \[
v_\gamma=\int_{\gamma}\omega
\]
which records the direction and the length of the saddle connection $\gamma$. We will denote the set of holonomy vectors corresponding to saddle connections in $(X,\omega)$ by $\Lambda_{sc}(X,\omega)$. Given a translation surface $(X,\omega)$ the set 
$\Lambda_{sc}(X,\omega)$
is a discrete subset of $\mathbb{R}^2$, with the property that $g\Lambda_{sc}(X,\omega)=\Lambda_{sc}g(X,\omega)$, see \cite{HS,Masur,Vorobets}.   

A translation surface $X$ comes with a set of topological data, \emph{the genus, the set of zeros, and the multiplicities of zeros}.  By the Gauss-Bonnet theorem the sum of the order of the zeros is equal to $2g-2$ where $g$ is the genus of the surface $X$. Therefore, the topological data can be represented by a vector $\vec{\alpha}=(\alpha_1,\dotsc,\alpha_m)$ where $\alpha_i$ is the order of the $i^{th}$ zero. We say that two translation surfaces are equivalent if there is an orientation preserving isometry between them which preserves the preferred vertical direction. Given a topological data $\vec{\alpha}$ the moduli space $\mathcal{H}(\vec{\alpha})$ of translation surfaces up to above equivalence relation is called a \emph{stratum}. 

There is a natural $SL(2,\mathbb{R})$ action on the space of translation surfaces: Given a translation surface $(X,\omega)$, which is a finite union of polygons $\{P_1,\dotsc,P_k\}$, and a matrix $A\in SL(2,\mathbb{R})$, $A\cdot(X,\omega)$ is the translation surface defined by the union of polygons $\{AP_1,\dotsc,AP_k\}$. It is easy to see that the $SL(2,\mathbb{R})$ action preserves the topological data, so it induces an action on each strata $\mathcal{H}(\vec{\alpha})$.    

The stabilizer of $(X,\omega)$ under the $SL(2,\mathbb{R})$ action on the space of translation surfaces is denoted by $SL(X,\omega)$. Equivelantly, $SL(X,\omega)$ is the the group of derivatives of the orientation preserving affine diffeomorphisms $(X,\omega)$. We will call the stabilizer $SL(X,\omega)$ the \emph{Veech group} of $(X,\omega)$. We note here that some authors call the image of $SL(X,\omega)$ in $\mathbb{P}SL(2,\mathbb{R})$, denoted by $\mathbb{P}SL(X,\omega)$, the Veech group, but we will stick to the notation of Smillie-Ulcigrai in \cite{SU}. 
A translation surface $(X,\omega)$ is called a \emph{Veech surface} if $SL(X,\omega)$ has finite covolume, i.e. $SL(2,\mathbb{R})/SL(X,\omega)$ has finite volume.


The image of the Veech group $\Gamma_{\mathcal{O}}$ of the octagon in $\mathbb{P}SL(2,\RR)$ is isomorphic to the triangle group $\Delta(4,\infty,\infty)$. The presence of multiple cusps gives rise to new computational challenges but also gives a conjectural idea of how to handle the general case where the Veech group has arbitrarily many cusps. Hence the computation of the gap distribution on $\mathcal{O}$ gives yet another step towards formulating and proving a general statement about gap distributions of saddle connections on any Veech surface. Indeed, in Theorem \ref{parametrizingsection} we give an effective way of parametrizing a Poincar\'e section to the horocycle flow on $SL(2,\RR)/SL(X,\omega)$ where the Veech group $SL(X,\omega)$ has arbitrarily many cusps. 

\subsection{Poincar\'e section for a general Veech surface}\label{generalVeech}

Let $(X,\omega)$ be a Veech surface, such that $SL(X,\omega)$ has $n<\infty$ cusps. We further assume here that the Veech group contains $-Id$; in Section \ref{poincaregeneral} we remove this assumption and prove a theorem for the general case.  It is well known that the set of saddle connections $\Lambda_{sc}(X,\omega)$ on $(X,\omega)$ can be decomposed into finitely many $SL(X,\omega)$ orbits of saddle connections, see \cite{HS}. Let $\Lambda_{sc}^{short}(X,\omega):=\sqcup_{i=1}^{n}SL(X,\omega)\cdot w_i\subset\Lambda_{sc}(X,\omega)$ be a set of shortest saddle connections in any given direction, so that any saddle connection on $(X,\omega)$ is parallel to exactly one element of $\Lambda_{sc}^{short}(X,\omega)$. Let us denote $\Lambda_{sc}^{w_i}(X,\omega)=SL(X,\omega)\cdot w_i$. Consider the set
\[
\Omega^{M}=\{gSL(X,\omega)\mid g(X,\omega)\text{ has a horizontal saddle connection of length $\le1$}\}.\]
According to \cite{Ath13}, $\Omega^{M}$ is a Poincar\'e section for the action of  horocycle flow on the moduli space $SL(2,\RR)/SL(X,\omega)$. By definition of  $\sqcup_{i=1}^{n}\Lambda_{sc}^{w_i}(X,\omega)$ we can write
\[
\Omega^{M} = \Omega_1^{M} \cup \cdots \cup \Omega_n^{M},\]
where
\[\Omega_i^{M} = \{gSL(X,\omega)\mid g\Lambda_{sc}^{w_i}(X)\cap(0,1]\neq\emptyset\},
\]
is the set of elements $gSL(X,\omega)$ in the $SL(2,\RR)$ orbit of the Veech group, such that $g\Lambda_{sc}^{w_i}(X)$ contains a horizontal saddle connection of length $\leq 1$.

Let $\Gamma_1,\dotsc,\Gamma_n$ be maximal parabolic sugroups of $SL(X,\omega)$ representing the conjugacy classes of all maximal parabolic subgroups in $SL(X,\omega)$. Each $\Gamma_i$ is isomorphic to $\mathbb{Z}\oplus\mathbb{Z}/2\mathbb{Z}$. Pick a generator $P_i$ for the infinite cyclic factor of $\Gamma_i$ with eigenvalue $1$. It follows from \cite{V89} that, each  $w_i$ in the above decomposition can be chosen so that $w_i$ is an eigenvector for $P_i$.  

Conjugate the parabolic elements $P_i\in\Gamma_i$ with an element $C_i$ in $SL(2,\RR)$ such that 
\[S_i=C_iP_iC_i^{-1} = \left[\begin{array}{cc}1 & \alpha_i \\ 0 & 1\end{array}\right].\] Moreover, $C_i$ can be chosen so that $C_i\cdot w_i=\left[\begin{array}{c}1\\0\end{array}\right]$. 
Define $M_{a,b}$ and sets $\Omega_i$ as follows
\[
M_{a,b}=\left[\begin{array}{cc} a & b \\ 0 &  a^{-1}\end{array}\right],\]
where $a\ne0 ,b$ are real numbers, and
\[\Omega_i:=\{(a,b)\in\mathbb{R}^{2}\mid 0<a\le1\ \text{and}\  1-(\alpha_i)a<b\le1\}.\]

\begin{thm}\label{parametrizingsection}
Let $(X,\omega)$ be a Veech surface such that the Veech group $SL(X,\omega)$ has $n<\infty$ cusps. There are coordinates from the section $\Omega^M$ to the set $\bigsqcup_{i=1}^{n}\Omega_i$. More precisely, the bijection between $SL(2,\RR)/SL(X,\omega)$ and $SL(2,\RR)\cdot (X,\omega)$ sends $\Omega_{i}^M$ to $\{M_{a,b}C_{i}(X,\omega)\mid(a,b)\in\Omega_i\}$, and the latter set is bijectively parametrized by $\Omega_i$. The return time function is piecewise rational with pieces defined by linear equations in these coordinates. The limiting gap distribution for any Veech surface is piecewise real analytic.
\end{thm}

\begin{ack*} 
The authors are grateful to their respective advisors Chris Leininger and Ilya Kapovich, and Jayadev Athreya for their constant support, advice and encouragement. We would like to thank Samuel Leli\`evre for his willingness to answer our questions and very useful discussions. We would like to thank J.D. Quigley for his aid in creating and debugging our sage code. We would also like to thank Zeev Rudnick for helpful comments on an earlier version of this paper. In addition, we thank the anonymous referee for his/her suggestions.

\end{ack*}


\section{Veech Groups and Gaps}

\subsection{The Octagon and the normalized $L$}\label{groups}

Recall that the image of the Veech group $\Gamma_{\mathcal{O}}$ of the octagon in $\mathbb{P}SL(2,\RR)$ is isomorphic to the triangle group $\Delta(4,\infty,\infty)$. The group $\Gamma_{\mathcal{O}}$, as calculated by  Smillie-Ulcigrai in \cite{SU}, is generated by the following two matrices:

\begin{displaymath}
\rho = \left [\begin{array}{c c}\frac{1}{\sqrt{2}} & \frac{-1}{\sqrt{2}} \\ \frac{1}{\sqrt{2}} & \frac{1}{\sqrt{2}}\end{array}\right ]
\hspace{1in}
\sigma = \left [\begin{array}{c c} 1 & 2(1 + \sqrt{2}) \\ 0 & 1\end{array}\right ]
\end{displaymath}

The matrix \[
T=\left[ \begin{array}{c c} 1 & 0 \\ 0 & \sqrt{2}\end{array}\right]\left[ \begin{array}{c c} 1 & (1 + \sqrt{2}) \\ 0 & 1\end{array}\right]=\left[ \begin{array}{c c} 1 & (1 + \sqrt{2}) \\ 0 & \sqrt{2}\end{array}\right]\in GL(2,\RR)\]
transforms $\mathcal{O}$ into $\mathcal{L}$, see Figure \ref{normalL}.  
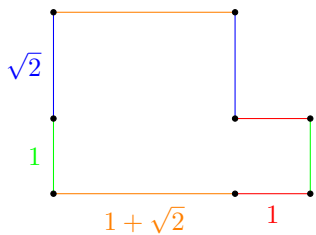
\begin{figure}[h!]
\centering
\begin{tikzpicture}
\coordinate (A) at (0,0);
\coordinate (B) at (0,2/2);
\coordinate (C) at (0,4.8282/2);
\coordinate (D) at (4.8282/2,4.8282/2);
\coordinate (E) at (4.8282/2,2/2);
\coordinate (F) at (6.8282/2,2/2);
\coordinate (G) at (6.8284/2,0);
\coordinate (H) at (4.8282/2,0);
\coordinate (I) at (0,.5);
\coordinate (J) at (6.8282/2,.5);
\draw [green] (A) -- (B) node[left=.25mm,pos=.5,color=green]{\small $1$};
\draw [blue] (B) -- (C) node[left=.25mm,pos=.5,color=blue]{\small $\sqrt{2}$};
\draw [orange](A) -- (H) node[below=.25mm,pos=.5,color=orange]{\small $1+\sqrt{2}$};
\draw [red](H) -- (G) node[below=.25mm,pos=.5,color=red]{\small $1$};
 \draw [orange] (C) -- (D);
 \draw [blue] (D) -- (E);
 \draw [red] (E) -- (F);
 \draw [green] (F) -- (G);
    \draw[fill] (B) circle (1pt);
    \draw[fill] (C) circle (1pt);
    \draw[fill] (D) circle (1pt);
    \draw[fill] (E) circle (1pt);
    \draw[fill] (F) circle (1pt);
        \draw[fill] (A) circle (1pt);
    \draw[fill] (G) circle (1pt);
    \draw[fill] (H) circle (1pt);
\end{tikzpicture}
\caption{The Veech surface $\mathcal{L}$}\label{normalL}
\end{figure}
We carry out the computations for the Veech surface $\mathcal{L}$.

Note that if the limiting gap distribution of a translation surface $X$ is given by
\[\lim_{R\rightarrow 0} \frac{1}{N(R)}|\mathbb{G}_X^R\cap (a,b)| = \int_a^b f(x) dx\]
then the limiting gap distribution for $cX$, with $c \in \RR$ will be given by
\[\lim_{R\rightarrow \infty}\frac{1}{N(cR)}|\mathbb{G}_{cX}^{cR} \cap (a,b)| = \frac{1}{c^4}\int_a^b f\left(\frac{x}{c^2}\right)dx.\]
In order to see this, we observe that scaling does not change the slopes, hence we have
\[\mathbb{S}_X^R = \mathbb{S}_{cX}^{cR} = \left\{0 \leq s_0^R < s_1^R < \cdots < s_{N(R)-1}^R\right\}\]
where we now have to consider slopes in a ball of radius $cR$ about the origin. Thus the set of slope gaps is renormalized by $(cR)^2$ so,
\[\mathbb{G}_{cX}^{cR} = \left\{(cR)^2(s_i^R - s_{i-1}^R : 1 \leq N(R) - 1, s_i \in \mathbb{S}_X^R\right\}.\]
Therefore we are interested in computing
\begin{align*}\lim_{R \rightarrow \infty} \frac{1}{N(cR)} \left|\mathbb{G}_{cX}^{cR} \cap (a,b)\right| &= \lim_{R\rightarrow \infty} \frac{1}{N(cR)}\cdot \frac{N(R)}{N(R)}\left|\{c^2R^2(s_i^R-s_{i-1}^R)\} \cap (a,b)\right|\\
&= \lim_{R\rightarrow\infty} \frac{1}{N(cR)}\cdot \frac{N(R)}{N(R)} \left|\{R^2(s_i^R-s_{i-1}^R)\}\cap(a/c^2, b/c^2)\right|\\
&=\frac{1}{c^2}\int_{a/c^2}^{b/c^2} f(x) dx
\end{align*}
Applying a change of variables then yields the desired result.

We can write the matrix $T$, which transforms $\mathcal{O}$ to $\mathcal{L}$, as $cg$, where $g \in SL(2, \mathbb{R})$, since
\[
2^{1/4}\left[\begin{array}{cc}\frac{1}{2^{1/4}} & 0 \\ 0 & \frac{\sqrt{2}}{2^{1/4}}\end{array}\right]=\left[\begin{array}{cc}1 & 0 \\ 0 & \sqrt{2}\end{array}\right].
\]
As we show in Section \ref{poincare}, the distribution of gaps is preserved under the $SL(2,\RR)$ action, hence the limiting gap distribution for $\mathcal{O}$ and $\mathcal{L}$ is related by the aforementioned formula with $c = \frac{1}{2^{1/4}}$. 

Conjugating the Veech group $\Gamma_{\mathcal{O}}$ of the octagon by the transformation $T$,  gives the following generators for the Veech group $\Gamma$ of $\mathcal{L}$:

\begin{displaymath}
R = \left [\begin{array}{c c} 1+\sqrt{2} & -(2 + \sqrt{2}) \\ \ 1 & -1  \end{array}\right ]
\hspace{1in}
S = \left [\begin{array}{c c} 1 & 2 + \sqrt{2} \\ 0 & 1\end{array}\right ]
\end{displaymath}
where $R$ is a finite order elliptic element and $S$ is an infinite order parabolic element. By building a fundamental domain, we also see that $S$ generates the stabilizier of the infinity in $\mathbb{P}SL(\mathcal{L})$.  

\subsection{Gap Distributions}
Let $\mathbb{S}_{\mathcal{L}}^R := \left\{0 = s_0^R < s_1^R < s_2^R < \cdots < s_{N(R)-1}^R\right\}$ be the set of slopes of saddle connections on the translation surface $\mathcal{L}$ that lie in the first quadrant and in the  ball of radius $R$ around the origin with respect to $\ell_{\infty}$ metric. Let \[
\mathbb{G}_{\mathcal{L}}^R := \left\{R^{2}(s_{i}^R - s_{i-1}^R) : 1 \leq i \leq N(R)-1, s_i \in \mathbb{S}_{\mathcal{L}}^R\right\}
\] be the set of renormalized slope gaps. Recall that the main goal of this paper to compute the following limit: 
\[
\lim_{R\rightarrow \infty} \frac{|\mathbb{G}_{\mathcal{L}}^R \cap (a,b)|}{N(R)}.
\]

To evaluate this limit we translate the gap distribution question into one about the return time of the horocycle flow to a Poincar\'e section. For the purpose of this paper, horocycle flow will be defined by the action of the subgroup
\[\left\{h_s = \left[\begin{array}{cc}1 & 0 \\ -s & 1\end{array}\right] : s \in \mathbb{R}\right\}.\]

This definition was chosen so that it acts on slopes by translations. That is, if we consider a vector $v$ with slope $\sigma$, then the slope of $h_s\cdot v$ will be $\sigma-s$.
The key step in the strategy is the construction of an appropriate Poincar\'e section for the horocycle flow.
That is, a set $\Omega$ such that under the horocycle flow the orbit of almost every point in moduli space intersects $\Omega$ in a countable discrete set of times. The cross section we choose is the set of translation surfaces in the $SL(2,\RR)$ orbit of the $\mathcal{L}$ with a short horizontal saddle connection. If we apply horocycle flow to one of these surfaces, we see that the next saddle connection to become horizontal, is the saddle connection with short horizontal component and smallest slope, and the return time to the Poincar\'e section is exactly its slope.

\section{Moduli Space, Poincar\'e Section, Computations}

We now prove Theorem \ref{octthm} generalizing a strategy used by Athreya, Cheung in \cite{AC13} and Athreya, Chaika, Leli\`evre in \cite{ACL}.

\subsection{Poincar\'e Section}\label{poincare}

It is well-known \cite{HS} that for a Veech surface $(X,\omega)$ the set $\Lambda_{sc}(X,w)$ of saddle connections can be written as a union of finitely many disjoint $SL(X,\omega)$ orbits of saddle connections. In other words, there are finitely many vectors $v_1,\dotsc,v_n\in\mathbb{R}^2$ such that \[
\Lambda_{sc}(X,\omega)=\bigcup_{i=1}^{n} SL(X,\omega)v_i.
\]
In our case, where the Veech surface $(X,\omega)$ is $\mathcal{L}$, we have 4 disjoint orbits:\[
\Lambda_{sc}(\mathcal{L})=\Gamma\left[\begin{array}{cc}1\\0\end{array}\right] \cup \Gamma\left[\begin{array}{cc}\sqrt{2}+1\\0\end{array}\right]\cup \Gamma\left[\begin{array}{cc}0\\1\end{array}\right]\cup\Gamma \left[\begin{array}{cc}0\\\sqrt{2}\end{array}\right]
\]
where the first two orbits and the last two orbits are pairwise parallel. Therefore, it suffices to consider only shorter orbits:
\[
\Lambda_{sc}^{h}(\mathcal{L}):=\Gamma\left[\begin{array}{cc}1\\0\end{array}\right],\  \Lambda_{sc}^{v}(\mathcal{L}):=\Gamma\left[\begin{array}{cc}0\\1\end{array}\right].
\]

It follows from the work of Athreya \cite{Ath13}, that the set
$\Omega^{M} =\{g\Gamma\mid g\Lambda_{sc}(\mathcal{L})\cap(0,1]\neq\emptyset\}$  is a Poincar\'e section for the action of the horocycle flow $h_s$ on the space $M=SL(2,\mathbb{R})\big/ \Gamma$. 
The observation above allows us to write $\Omega^{M}=\Omega_{h}^{M}\cup\Omega_{v}^{M}$ where
\begin{align*}
&\Omega_{h}^{M}=\{g\Gamma\mid g\Lambda_{sc}^{h}(\mathcal{L})\cap(0,1]\neq\emptyset\}\\
&\Omega_{v}^{M}=\{g\Gamma\mid g\Lambda_{sc}^{v}(\mathcal{L})\cap(0,1]\neq\emptyset\}. 
\end{align*} 

For a Veech surface $(X,\omega)$ the $SL(2,\RR)$ orbit of $(X,\omega)$ can be identified with the moduli space $SL(2,\RR)/SL(X,\omega)$. In what follows, and in Section \ref{poincaregeneral}, we will implicitly use this identification. Hence, in this setting, $\Omega^{M}=\Omega_{h}^{M}\cup\Omega_{v}^{M}$ is the set of translation surfaces in the $SL(2,\mathbb{R})$ orbit of $\mathcal{L}$ with a horizontal saddle connection of length $\le1$.

Now, we will give a parametrization of the section $\Omega^{M}$ in $\mathbb{R}^2$. We first need the following matrix: Let \[
M_{a,b}=\left[\begin{array}{cc} a & b \\ 0 &  a^{-1}\end{array}\right],\]
where $a\ne0 ,b$ are real numbers. 

Next define the following subsets of $\mathbb{R}^{2}$:
\begin{align*}
&\Omega_1=\{(a,b)\in\mathbb{R}^{2}\mid 0<a\le1\ \text{and}\  1-(\sqrt{2}+2)a<b\le1\} \\
&\Omega_2=\{(a,b)\in\mathbb{R}^{2}\mid 0<a\le1\ \text{and}\ 1-a<b\le1\}
\end{align*}

Let $\mathcal{L}^{R}=R_{\pi/2}\mathcal{L}=\left[\begin{smallmatrix} 0 & 1 \\ {-1} & 0\end{smallmatrix}\right]\mathcal{L}$ be the translation surface obtained from $\mathcal{L}$ by a rotation by $\pi/2$ in the clockwise direction. The Veech group $\Gamma_{R}$ of $\mathcal{L}^{R}$ can be obtained by conjugating the Veech group of $\mathcal{L}$. Hence $\Gamma_{R}$ is generated by

\begin{displaymath}
R_{\pi/2}(R)R^{-1}_{\pi/2} = \left [\begin{array}{c c} 1+\sqrt{2} & -1 \\ \ 2+\sqrt{2} & -1  \end{array}\right ]
\hspace{1in}
R_{\pi/2}(-RS^{-1})R^{-1}_{\pi/2} = \left [\begin{array}{c c} 1 & 1 \\ 0 & 1\end{array}\right ].
\end{displaymath}

\begin{thm}\label{Lgaps} There are coordinates from the section $\Omega_{h}^M\cup\Omega_{v}^{M}$ to the set $\Omega_1\cup\Omega_2$. More precisely, the canonical bijection  $\Theta:SL(2,\RR)/\Gamma\to SL(2,\RR)\cdot\mathcal{L}$ sends $\Omega_{h}^M$ to $\{M_{a,b}\mathcal{L}\mid(a,b)\in\Omega_1\}$ and $\Omega_{v}^M$ to $\{M_{a,b}\mathcal{L}^R\mid (a,b)\in\Omega_2\}$, and the latter sets are bijectively parametrized by $\Omega_1$ and $\Omega_2$. In these coordinates, the return time function $R:\Omega_1\cup\Omega_2\to\mathbb{R}_{>0}$, defined by 
\[
R(a,b)=\left\{\begin{array}{ll} \min\{s\mid h_sM_{a,b}\mathcal{L}\in\Omega_{h}^M\cup\Omega_{v}^M\} & \text{for $(a,b)\in\Omega_1$}\\\min\{s\mid h_sM_{a,b}\mathcal{L}^R\in\Omega_{h}^M\cup\Omega_{v}^M\} & \text{for $(a,b)\in\Omega_2$}\end{array}\right.
\] is a piecewise rational function with five pieces, which is uniformly bounded below by $1$. The return map $T:\Omega_1\cup\Omega_2\to\Omega_1\cup\Omega_2$ is a measure preserving bijection and it is piecewise linear with countably many pieces.    
\end{thm}

\begin{figure}[h!]
\centering
\begin{minipage}{.3\textwidth}
  \centering
  \includegraphics[width=.55\linewidth]{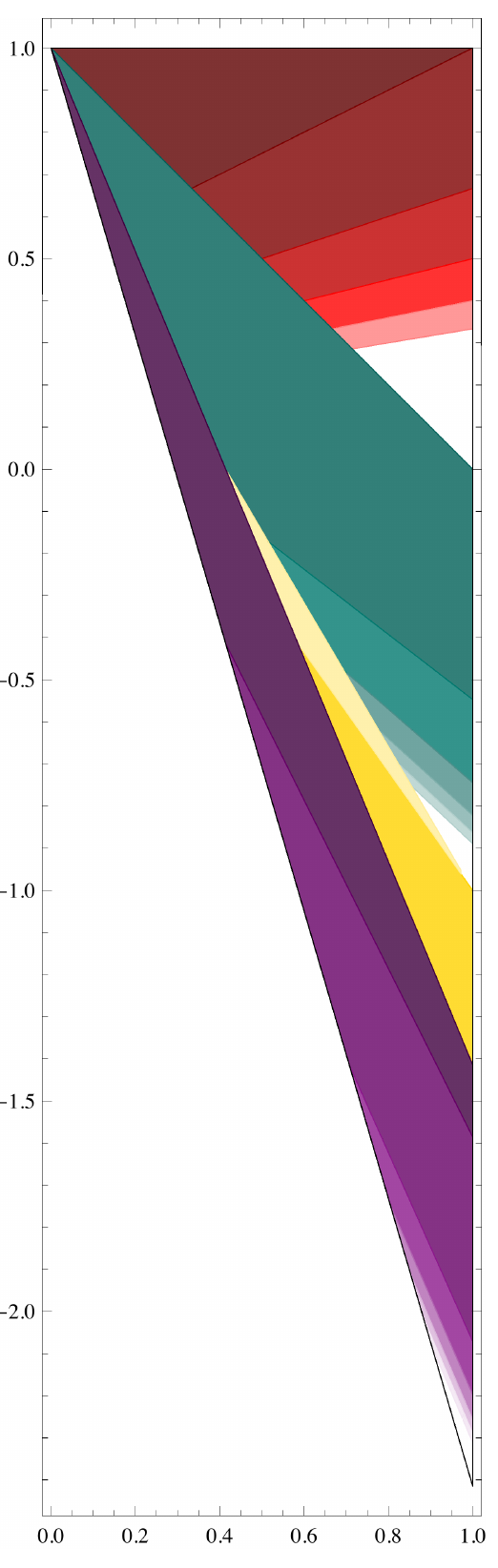}
  \label{O1S}
\end{minipage}%
\begin{minipage}{.3\textwidth}
  \centering
  \includegraphics[width=.55\linewidth]{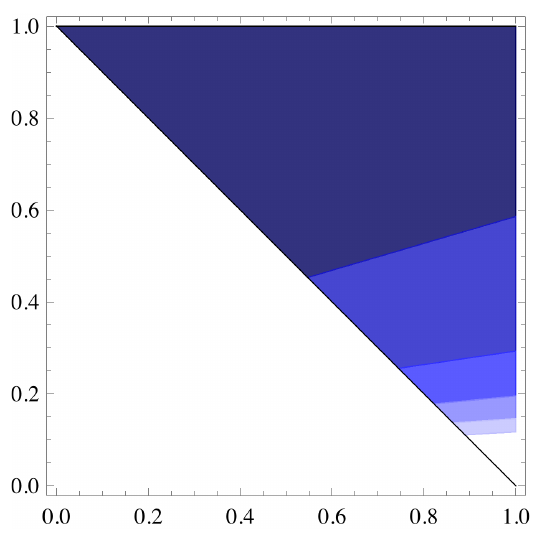}
  \label{O2S}
\end{minipage}
\caption{Source regions for the return map. Colors represent finer subdivisions corresponding to different values of the constant in each equation. White areas indicate infinitely many subdivisions following the same pattern as the others in the region.}
\label{src}
\end{figure}

\begin{figure}[h!]
\centering
\begin{minipage}{.3\textwidth}
  \centering
  \includegraphics[width=.55\linewidth]{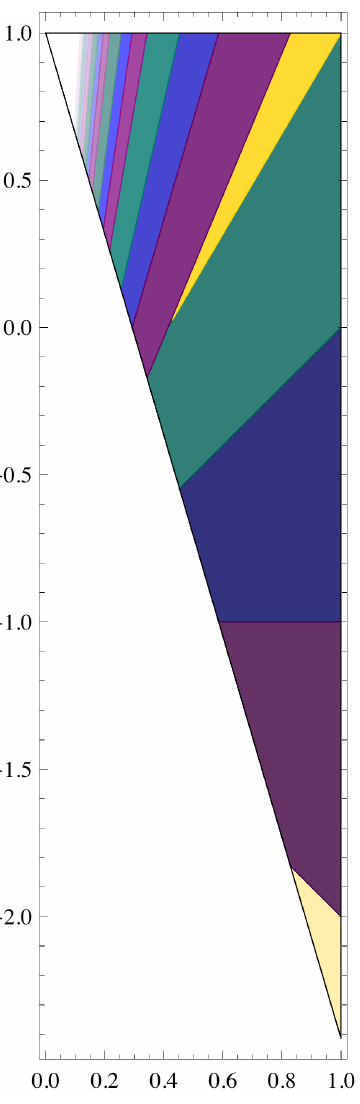}
  \label{O1T}
\end{minipage}%
\begin{minipage}{.3\textwidth}
  \centering
  \includegraphics[width=.55\linewidth]{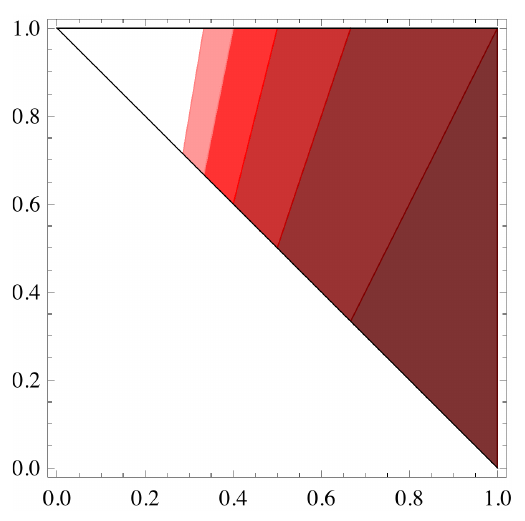}
  \label{O2T}
\end{minipage}
\caption{Target regions for the return map. Each colored region from the source maps to the same colored region in the target.}
\label{trgt}
\end{figure}

\begin{proof}  Suppose that $g\Gamma\in\Omega_{h}^{M}$. We need to show that $g\mathcal{L}=M_{a,b}\mathcal{L}$ for some $(a,b)\in\Omega_1$. By definition $g\mathcal{L}$ has a horizontal saddle connection $\left[\begin{array}{cc}a\\0\end{array}\right]$ such that 
\[
\left[\begin{array}{cc}a\\0\end{array}\right] =g\gamma\left[\begin{array}{cc}1\\0\end{array}\right]
\] for some $\gamma\in\Gamma,$ and $0<a\le1$.
Since $\Gamma$ acts on $\Lambda_{sc}^{h}(\mathcal{L})$ transitively and for any element $\gamma\in\Gamma$ we have $g\gamma \mathcal{L}$=$g\mathcal{L}$, we can assume that
\[
\left[\begin{array}{cc}a\\0\end{array}\right]=g\left[\begin{array}{cc}1\\0\end{array}\right].
\]
The elements in $SL(2,\RR)$ which take $\left[\begin{array}{cc}1\\0\end{array}\right]$ to $\left[\begin{array}{cc}a\\0\end{array}\right]$ are precisely of the form \[
M_{a,b}=\left [\begin{array}{c c} a & b \\ 0 & a^{-1}\end{array}\right].
\]
Thus, every element in  $\Omega_{h}^M$ can be written as $M_{a,b}\mathcal{L}$ for some $(a,b)$ where $0<a\le1$. 
Since the element $S=\left [\begin{array}{c c} 1 & 2 + \sqrt{2} \\ 0 & 1\end{array}\right]$ is in the Veech group $\Gamma$, and \[
M_{a,b}S^{-n}=\left[\begin{array}{c c} a & b-(2+\sqrt{2})an \\ 0 & a^{-1}\end{array}\right],
\]
we can write every element $g\mathcal{L}$ in $\Omega_{h}^M$ as $M_{a,b}\mathcal{L}$ where $(a,b)\in\Omega_1$. 

For any point $(a,b)\in\Omega_1$, it is clear that $M_{a,b}\mathcal{L}$ has a short horizontal, the image of the saddle connection $\left[\begin{array}{cc}1\\0\end{array}\right]$.

Finally, we prove that $\Omega_1$ bijectively parametrizes $\{M_{a,b}\mathcal{L}\mid (a,b)\in\Omega_1\}$. We suppose that $(a,b),(c,d)\in\Omega_1$ and that $M_{a,b}\mathcal{L}=M_{c,d}\mathcal{L}$, and must prove that $(a,b)=(c,d)$. We have \[
M_{a,b}M_{c,d}^{-1} = \left[\begin{array}{cc}\frac{a}{c} & bc-ad \\ 0 & \frac{c}{a}\end{array}\right] \in \Gamma,
\] which fixes infinity in the upper half-plane. As we noted in Section \ref{groups}, the stabilizer of infinity in $PSL(\mathcal{L})$ is the infinite cylic group generated by the matrix \[
S=\left [\begin{array}{c c} 1 & 2 + \sqrt{2} \\ 0 & 1\end{array}\right].\]
Since $a/c>0$ it follows that $a=c$. From this we see that $1-(2+\sqrt{2})a < b,d \leq 1$, and $(b-d)a = k (2+\sqrt{2})$, which implies $b = d$, and hence $(a,b) = (c,d)$, as required.\\
 
Now suppose that $g\Gamma\in\Omega_v^{M}$. We need to show that $g\mathcal{L}=M_{a,b}\mathcal{L}^{R}$ for some $(a,b)\in\Omega_2$. By definition, there is a horizontal saddle connection $\left[\begin{array}{cc}a\\0\end{array}\right]$ on $g\mathcal{L}$ such that \[
\left[\begin{array}{cc}a\\0\end{array}\right]=g\gamma\left[\begin{array}{cc}0\\1\end{array}\right].\]
As in the previous case, we can assume that \[
\left[\begin{array}{cc}a\\0\end{array}\right]=g\left[\begin{array}{cc}0\\1\end{array}\right].\]
This implies that \[
R_{\pi/2}^{-1}g\left[\begin{array}{cc}0\\1\end{array}\right]=\left[\begin{array}{cc}0\\a\end{array}\right].\]
The elements in $SL(2,\RR)$ that take $\left[\begin{array}{cc}0\\1\end{array}\right]$ to $\left[\begin{array}{cc}0\\a\end{array}\right]$ are precisely of the form $\left[\begin{array}{c c} a^{-1} & 0 \\ b & a\end{array}\right]$. This means that 
\[
g\mathcal{L}=R_{\pi/2}\left[\begin{array}{c c} a^{-1} & 0 \\ b & a\end{array}\right]\mathcal{L}=R_{\pi/2}\left[\begin{array}{c c} a^{-1} & 0 \\ b & a\end{array}\right]R_{\pi/2}^{-1}R_{\pi/2}\mathcal{L}=\left[\begin{array}{c c} a & -b \\ 0 & a^{-1}\end{array}\right]\mathcal{L}^{R},
\] 
which implies that every element $g\mathcal{L}\in\Omega_{2}^{M}$ can be written as $M_{a,b}\mathcal{L}^{R}$ for $0<a\le1$. Observe that the matrix \[
R_{\pi/2}(-RS^{-1})R^{-1}_{\pi/2}=\left[\begin{array}{c c} 1 & 1 \\ 0 & 1\end{array}\right]\in\Gamma_{R}=R_{\pi/2}\Gamma R^{-1}_{\pi/2},
\] the Veech group of $\mathcal{L}^{R}$. Thus, arguing as in the first case,  we can write every element $g\mathcal{L}\in\Omega_{2}^{M}$ as $M_{a,b}\mathcal{L}^{R}$ for $(a,b)\in\Omega_2$. The bijection between $\Omega_2$ and $\{M_{a',b'}\mathcal{L}^R\mid (a',b')\in\Omega_2\}$ now follows from a similar argument as in the case of $\Omega_1$. 

We now compute the return time function $R:\Omega_1\cup\Omega_2\to\RR_{>0}$ explicitly. Recall that for a given point $M_{a,b}\mathcal{L}$ the return time to the section under the horocycle flow is given by the smallest positive slope of the saddle connection on $M_{a,b}\mathcal{L}$ with horizontal component $\le1$. Therefore we must determine the associated slope for each point $(a,b) \in \Omega$. This requires finding the original saddle connection whose image has the smallest positive slope.

On the original $\mathcal{L}$ there are 4 different saddle connections of interest, shown in Figure \ref{sc}. We determined the set of points $(a,b)$ such that under the matrix $M_{a,b}$ each of the four saddle connections of interest has horizontal component less than $1$ and greater than $0$, in other words, that it is a candidate to have the smallest slope. This divides $\Omega_1$ into four subregions, these are shown in Figure \ref{subregions}.
\begin{figure}[h!]
\begin{tikzpicture}
\coordinate (A) at (0,0);
\coordinate (B) at (0,2);
\coordinate (C) at (0,4.8284);
\coordinate (D) at (4.8284,4.8284);
\coordinate (E) at (4.8284,2);
\coordinate (F) at (6.8284,2);
\coordinate (G) at (6.8284,0);
\coordinate (H) at (4.8284,0);
\draw (B) -- (C) node[left=.25mm,pos=1]{\small $G$} -- (D) -- (E) -- (F) -- (G) -- (H) node[right=.25mm,pos=0]{\small $C$} -- (A);
\draw [AH] (A) -- (B);
\draw [BD] (H) -- (F) node[below=.25mm,pos=0,color=black]{\small $B$} node[right=.25mm,pos=1,color=black]{\small $D$};
\draw [HF] (B) -- (4.8284, 4.8284) node[left=.25mm,pos=0,color=black]{\small $H$} node[right=.25mm,pos=1,color=black]{\small $F$};
\draw[AE] (A) -- (4.8284, 2) node[above=-.75mm,pos=1.06,color=black]{\small $E$};
    \draw[fill] (B) circle (1pt);
    \draw[fill] (C) circle (1pt);
    \draw[fill] (D) circle (1pt);
    \draw[fill] (E) circle (1pt);
    \draw[fill] (F) circle (1pt);
        \draw[fill] (A) circle (1pt) node[left=.25mm,pos=0]{\small $A$};
    \draw[fill] (G) circle (1pt);
    \draw[fill] (H) circle (1pt);
\end{tikzpicture}
\caption{The saddle connections of interest}
\label{sc}
\end{figure}
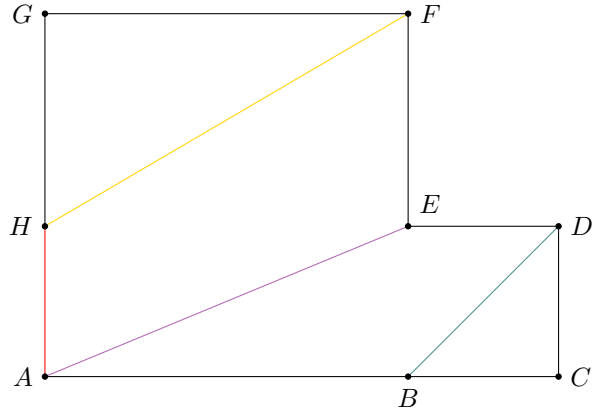

\begin{figure}[h!]

\centering
\includegraphics[scale=0.75]{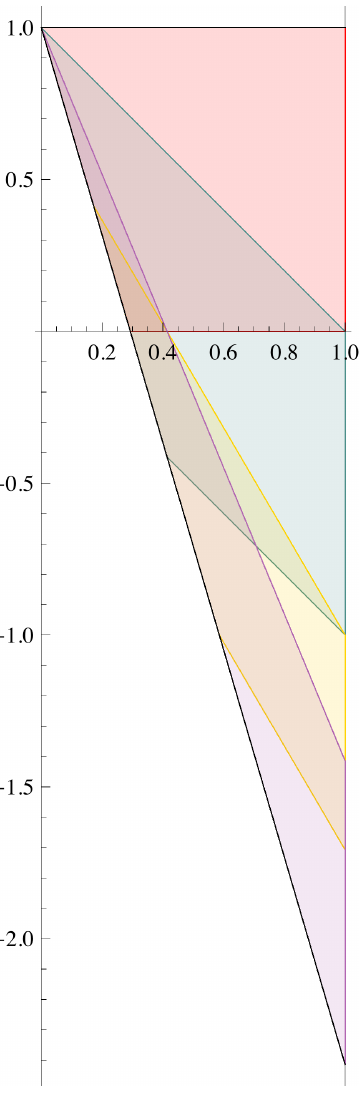}
\caption{Subregions}
\label{subregions}
\end{figure}
 
On the overlap of the red and green regions, the saddle connections coming from $AH$ and $BD$ are possible candidates. To check which one dominates, we consider their slopes. The slope of $M_{a,b}\cdot AH$ is $\frac{1}{ab}$ and the slope of  $M_{a,b}\cdot BD$ is $\frac{1}{a(a+b)}$. If $a,b > 0$, we have $a+b > b$ so $a(a+b) > ab$ and therefore $\frac{1}{a(a+b)}< \frac{1}{ab}$ which implies that the image of $BD$ dominates the image of $AH$, so $AH$ must end where $BD$ appears, at the line $b = 1-a$. 

On the overlap of the green and yellow regions, we see the saddle connections coming from $BD$ and $HF$. Again comparing slopes, we have the slope of $M_{a,b}\cdot HF$ is $\frac{\sqrt{2}}{a(a(\sqrt{2}+1)+b\sqrt{2})}$. Since
\[a>0, a + b + \frac{a}{\sqrt{2}} > a + b \text{ and so } \frac{1}{a\left(a+b+\frac{a}{\sqrt{2}}\right)} < \frac{1}{a(a+b)},\]
 therefore $HF$ dominates $BD$.

On the overlap of the green and purple regions, we see the saddle connections coming from $BD$ and $AE$. The slope of $M_{a,b}\cdot AE$ is $\frac{1}{a((1+\sqrt{2})a + b)}$. Since $a>0$, we have
\[a + \sqrt{2}a + b > a + b \text{ so } \frac{1}{a((1+\sqrt{2})a+b)} < \frac{1}{a(a+b)},\]
 thus $AE$ dominates $BD$. 

On the overlap of purple and yellow we see the saddle connections coming from $AE$ and $HF$. Since $a > 0$, we have
\[a + \sqrt{2}a + b > a + b + \frac{a}{\sqrt{2}} \text{ so } \frac{1}{a + \sqrt{2}a + b } < \frac{1}{a + b + \frac{a}{\sqrt{2}}}\]
 and thus $AE$ dominates $HF$.

One can show, using similar arguments and direct computation, that these four saddle connections are the only ones whose image has the smallest slope. 

This analysis results in the 4 subdivisions of $\Omega_1$ shown below in Figure \ref{cusp1}.  Each subsection is labeled in a way that for each point (a,b) the saddle connection with smallest slope and short horizontal component in $M_{a,b}\mathcal{L}$ comes from the saddle connection with the corresponding label. 

A direct computation also shows that for $\Omega_2$ there is no further subdivison, and the image of the short vertical saddle connection on $\mathcal{L}^{R}$ has the smallest slope in $M_{a,b}\mathcal{L}^{R}$ for every $(a,b)\in\Omega_2$.

\begin{figure}[h!]
\centering
\begin{minipage}{.3\textwidth}
  \centering
  \labellist
\small\hair 1pt
 \pinlabel {$\Omega_{AH}$} [ ] at 180 680
 \pinlabel {$\Omega_{BD}$} [ ] at 183 590
 \pinlabel {$\Omega_{HF}$} [ ] at 202 515
 \pinlabel {$\Omega_{AE}$} [ ] at 187 477
\endlabellist
\centering
\includegraphics[width=.65\linewidth]{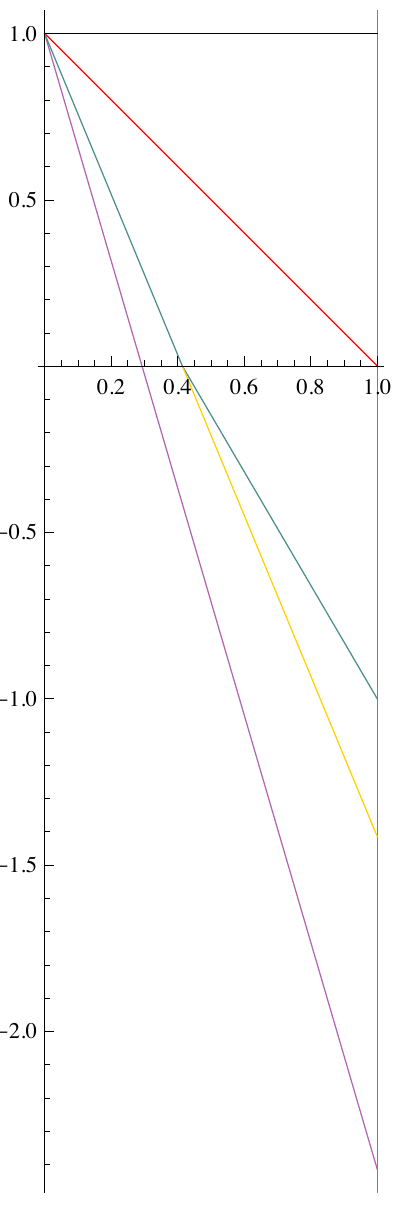}
\caption{$\Omega_1$}
\label{cusp1}
\end{minipage}%
\begin{minipage}{.3\textwidth}
\centering
 \labellist
\small\hair 1pt
 \pinlabel {$\Omega_{2}$} [ ] at 230 260
\endlabellist
\centering
\includegraphics[width=.55\linewidth]{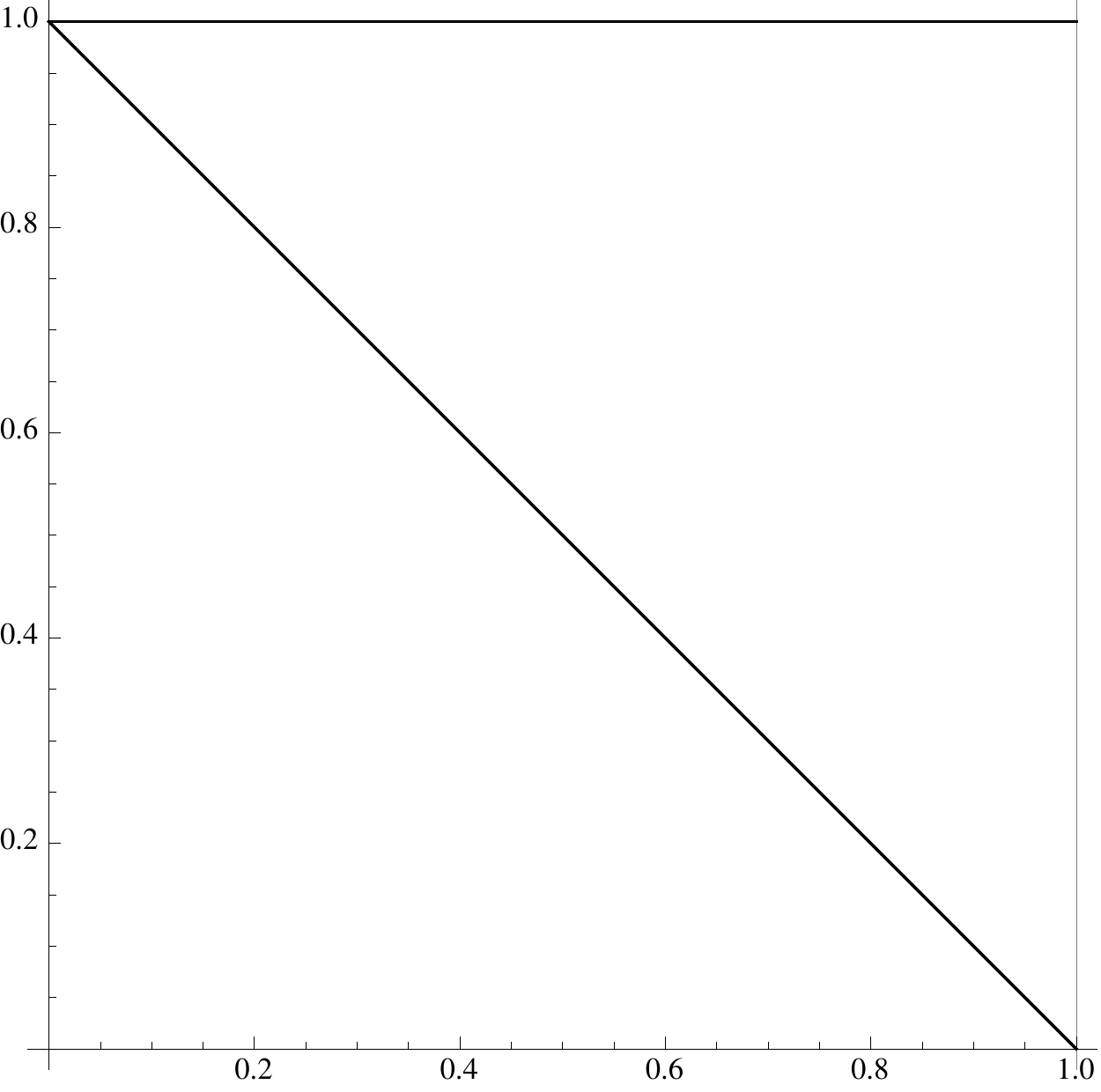}
\caption{$\Omega_2$}
\label{cusp2}
\end{minipage}
\end{figure}

Therefore, the return time function in these coordinates is given by:
\[
R(a,b)=\left\{\begin{array}{ll} 
\frac{1}{ab} & \text{for $(a,b)\in\Omega_{AH}$}\\
\frac{1}{a(a+b)} & \text{for $(a,b)\in\Omega_{BD}$}\\
\frac{\sqrt{2}}{a(a(1+\sqrt{2})+b\sqrt{2})} & \text{for $(a,b)\in\Omega_{HF}$}\\
\frac{1}{a(a(1+\sqrt{2})+b)} & \text{for $(a,b)\in\Omega_{AE}$}\\
\frac{1}{ab} & \text{for $(a,b)\in\Omega_{2}$}\\
\end{array}\right.
\] 

The strategy for finding the return map is as follows: Given a point $(a,b)\in\Omega_1$, the image is $(a',b')$, where either $(a',b')\in\Omega_1$ is such that  $h_{R(a,b)}M_{a,b}\mathcal{L}=M_{a',b'}\mathcal{L}$ or $(a',b')\in\Omega_{2}$ such that $h_{R(a,b)}M_{a,b}\mathcal{L}=M_{a',b'}\mathcal{L}^{R}$. Similarly, for any point $(a,b)\in\Omega_2$, the image is $(a',b')$ where either $(a',b')\in\Omega_1$ is such that  $h_{R(a,b)}M_{a,b}\mathcal{L}^{R}=M_{a',b'}\mathcal{L}$ or $(a',b')\in\Omega_{2}$ such that $h_{R(a,b)}M_{a,b}\mathcal{L}^{R}=M_{a',b'}\mathcal{L}^{R}$, in both cases $h_{R(a,b)}$ is the horocyle flow at the time $s=R(a,b)$. 

Note that \[
h_sM_{a,b}=\left[\begin{array}{cc} a & b \\ -sa & -sb + a^{-1}\end{array}\right].\] 
In what follows we apply an element of $\Gamma=SL(\mathcal{L})$ to $\mathcal{L}$ without changing the equality so that we can write our point in the required form. 

\begin{enumerate}

\item[$\Omega_{AH}$:] In this subregion we have $R(a,b)=\frac{1}{ab}$, and hence
$h_{R(a,b)}M_{a,b} = \left[\begin{array}{cc} a & b \\ -\frac{1}{b} & 0\end{array}\right]$.\\ 

So we have, 
\begin{align*}
&\left[\begin{array}{cc} a & b \\ -\frac{1}{b} & 0\end{array}\right] \left[\begin{array}{cc} 0 & -1 \\ {1} & 0\end{array}\right]\left[\begin{array}{cc} 0 & 1 \\ {-1} & 0\end{array}\right]\mathcal{L}  =\\ &\left[\begin{array}{cc} b & -a  \\ 0 & \frac{1}{b}\end{array}\right]\mathcal{L}^{R}=\left[\begin{array}{cc} b & -a  \\ 0 & \frac{1}{b}\end{array}\right]\left[\begin{array}{cc} 1 & 1  \\ 0 & 1\end{array}\right]^{k_{AH}}\mathcal{L}^{R}=\left[\begin{array}{cc} b & -a+bk_{AH}  \\ 0 & \frac{1}{b}\end{array}\right]\mathcal{L}^{R}
\end{align*} where $k_{AH}=\left\lfloor\frac{1+a}{b}\right\rfloor$, and $\left[\begin{array}{cc} 1 & 1  \\ 0 & 1\end{array}\right]\in\Gamma^{R}.$

\item [$\Omega_{BD}$:] In this subregion we have $R(a,b)=\frac{1}{a(a+b)}$, therefore $h_{R(a,b)}M_{a,b}=\left[\begin{array}{cc} a & b \\ -\frac{1}{a+b} & \frac{1}{a+b}\end{array}\right]$. 

So we have,
\begin{align*}
&\left[\begin{array}{cc} a & b \\ -\frac{1}{a+b} & \frac{1}{a+b}\end{array}\right]\mathcal{L}=\left[\begin{array}{cc} a & b \\ -\frac{1}{a+b} & \frac{1}{a+b}\end{array}\right]\left[\begin{array}{cc} 1 & 0 \\ 1 & 1\end{array}\right]\mathcal{L} =\\
&\left[\begin{array}{cc} a+b & b \\ 0 & \frac{1}{a+b}\end{array}\right]\left[\begin{array}{cc} 1 & 2+\sqrt{2} \\ 0 & 1\end{array}\right]^{k_{BD}}\mathcal{L}=\left[\begin{array}{cc} a + b & b + (2+\sqrt{2})(a+b)k_{BD} \\ 0 & \frac{1}{a+b}\end{array}\right]\mathcal{L}
\end{align*}
where 
$k_{BD} = \left\lfloor\frac{1-b}{(2+\sqrt{2})(a+b)}\right\rfloor$, and $\left[\begin{array}{cc} 1 & 0 \\ 1 & 1\end{array}\right],\left[\begin{array}{cc} 1 & 2+\sqrt{2} \\ 0 & 1\end{array}\right]\in\Gamma$. 

\item[$\Omega_{HF}$:] In this subregion $R(a,b)=\frac{\sqrt{2}}{a(a(1+\sqrt{2})+b\sqrt{2})}$, and hence $h_{R(a,b)}M_{a,b}=\left[\begin{array}{cc} a & b \\ -\frac{\sqrt{2}}{a(1+\sqrt{2})+b\sqrt{2}} & \frac{1+\sqrt{2}}{a(1+\sqrt{2})+b\sqrt{2}}\end{array}\right]$

\begin{eqnarray*}&&\left[\begin{array}{cc} a & b \\ -\frac{\sqrt{2}}{a(1+\sqrt{2}) + b\sqrt{2}} & \frac{1+\sqrt{2}}{a(1+\sqrt{2}) + b\sqrt{2}}\end{array}\right] \left[\begin{array}{cc} 1 + \sqrt{2} & \sqrt{2}+2 \\ \sqrt{2} & 1+\sqrt{2}\end{array}\right]\left[\begin{array}{cc} 1& \sqrt{2}+2 \\ 0 & 1\end{array}\right]^{k_{HF}}\mathcal{L}=
\\&& \left[\begin{array}{c c} (1+\sqrt{2})a + \sqrt{2}b & (2+\sqrt{2})a + (1+\sqrt{2})b + ((1+\sqrt{2})a + \sqrt{2}b)(2+\sqrt{2})k_{HF}\\0 & \frac{1}{(1+\sqrt{2})a + \sqrt{2}b}\end{array}\right]\mathcal{L}
\end{eqnarray*}
where $k_{HF} = \left\lfloor \frac{1-((2+\sqrt{2})a + (1+\sqrt{2})b)}{(2+\sqrt{2})((1+\sqrt{2})a + \sqrt{2}b)}\right\rfloor$, and $\left[\begin{array}{cc} 1 + \sqrt{2} & \sqrt{2}+2 \\ \sqrt{2} & 1+\sqrt{2}\end{array}\right],\left[\begin{array}{cc} 1& \sqrt{2}+2 \\ 0 & 1\end{array}\right]\in\Gamma$. 

\item[$\Omega_{AE}$:] In this subregion $R(a,b)=\frac{1}{a(a(1+\sqrt{2})+b)}$, and hence $h_{R(a,b)}M_{a,b}=\left[\begin{array}{cc} a & b \\ -\frac{1}{(1+\sqrt{2})a + b} & \frac{1+\sqrt{2}}{(1+\sqrt{2})a + b}\end{array}\right]$. 

So we have,
 
\begin{eqnarray*}
&&\left[\begin{array}{cc} a & b \\ -\frac{1}{(1+\sqrt{2})a + b} & \frac{1+\sqrt{2}}{(1+\sqrt{2})a + b}\end{array}\right] \left[\begin{array}{cc} 1+\sqrt{2} & 2(1+\sqrt{2}) \\ 1 & 1+\sqrt{2}\end{array}\right]\left[\begin{array}{cc} 1& \sqrt{2}+2 \\ 0 & 1\end{array}\right]^{k_{AE}}\mathcal{L} =\\ && \left[\begin{array}{cc}(1+\sqrt{2})a+b & 2(1+\sqrt{2})a+(1+\sqrt{2})b +((1+\sqrt{2})a+b)(2+\sqrt{2})k_{AE} \\ 0 & \frac{1}{(1+\sqrt{2})a + b}\end{array}\right]\mathcal{L}
\end{eqnarray*}
where $k_{AE} = \left\lfloor \frac{1-(2(1+\sqrt{2})a+(1+\sqrt{2})b)}{(2+\sqrt{2})((1+\sqrt{2})a+b)}\right\rfloor.$

\bigskip

\item[$\Omega_2$:] For this region $R(a,b)=\frac{1}{ab}$ and hence $h_{R(a,b)}M_{a,b} = \left[\begin{array}{cc} a & b \\ -\frac{1}{b} & 0\end{array}\right]$. 

So we have,

\begin{align*}&
\left[\begin{array}{cc} a & b \\ -\frac{1}{b} & 0\end{array}\right]L^{R}=\left[\begin{array}{cc} a & b \\ -\frac{1}{b} & 0\end{array}\right] \left[\begin{array}{cc} 0 & 1 \\ -1 & 0\end{array}\right]\mathcal{L}=\\
&\left[\begin{array}{cc} -b & a \\ 0 & -\frac{1}{b}\end{array}\right]\left[\begin{array}{cc} -1 & 0 \\ 0 & -1\end{array}\right]\left[\begin{array}{cc} 1& \sqrt{2}+2 \\ 0 & 1\end{array}\right]^{k_{2}}\mathcal{L}=\\
&\left[\begin{array}{cc}b & -a + (2 + \sqrt{2})bk_{2}\\ 0 & -\frac{1}{b}\end{array}\right]\mathcal{L}, 
\end{align*}
where $k_{2} = \left\lfloor \frac{1+a}{(2+\sqrt{2})b}\right\rfloor$ and $\left[\begin{array}{cc} -1 & 0 \\ 0 & -1\end{array}\right], \left[\begin{array}{cc} 1& \sqrt{2}+2 \\ 0 & 1\end{array}\right]\in\Gamma$. 

\end{enumerate}

Thus the return map in these coordinates can be defined on each subregion as follows:
\begin{align*}
&T_{AH}(a,b)=
(b,-a+bk_{AH}) \in \Omega_2,
\text{where }k_{AH}=\left\lfloor\frac{1+a}{b}\right\rfloor\\
&T_{BD}(a,b)=
(a+b,b+(2+\sqrt{2})(a+b)k_{BD}) \in \Omega_1,
\text{where }k_{BD} = \left\lfloor\frac{1-b}{(2+\sqrt{2})(a+b)}\right\rfloor\\
&T_{HF}(a,b)=
((1+\sqrt{2})a + \sqrt{2}b,(2+\sqrt{2})a + (1+\sqrt{2})b + ((1+\sqrt{2})a + \sqrt{2}b)(2+\sqrt{2})k_{HF}) \in \Omega_1,\\
&\hspace{60pt}\text{where }k_{HF} = \left\lfloor \frac{1-((2+\sqrt{2})a + (1+\sqrt{2})b)}{(2+\sqrt{2})((1+\sqrt{2})a + \sqrt{2}b)}\right\rfloor\\
&T_{AE}(a,b)=
((1+\sqrt{2})a+b,2(1+\sqrt{2})a+(1+\sqrt{2})b +((1+\sqrt{2})a+b)(2+\sqrt{2})k_{AE}) \in \Omega_1,\\
&\hspace{60pt}\text{where }k_{AE} = \left\lfloor \frac{1-(2(1+\sqrt{2})a+(1+\sqrt{2})b)}{(2+\sqrt{2})((1+\sqrt{2})a+b)}\right\rfloor\\
&T_{\Omega_2}(a,b)=
(b, -a+(2+\sqrt{2})bk_2) \in \Omega_1,
\text{where }k_2 =\left\lfloor \frac{1+a}{(2+\sqrt{2})b}\right\rfloor.
\end{align*}

\end{proof}

A picture of the return map can be found in Figures \ref{src}, \ref{trgt}. Figure \ref{src} depicts the various subregions corresponding to different values of the constants $k_{AH}$, $k_{BD}$, $k_{HF}$, $k_{AE}$, and $k_2$. Based on the return map function these regions map to the region of the same color in Figure \ref{trgt}. As you can see, the blue regions from $\Omega_2$ all map into $\Omega_1$, the red regions from $\Omega_{AH}$ map into $\Omega_2$, and the green, yellow, and purple regions all remain in $\Omega_1$.

\subsection{Slope gaps in the vertical strip}\label{verticalstrip}
Let $g\mathcal{L}$ be an element in the $SL(2,\RR)$ orbit of $\mathcal{L}$. Let 
\[
\overline{\mathbb{S}}_{g\mathcal{L}}=\{0<s_0<s_1<\dotsc<s_{N+1}<\dotsc\}
\]
be the set of positive slopes of saddle connections on $g\mathcal{L}$ in the vertical strip $V_1$ where
\[
V_1=\{(0,1]\times[0,\infty)\}\subset \RR^2.
\]

Since $s_0$ is the smallest positive slope in the vertical strip, the first time $h_{s}g\mathcal{L}$ hits the section is when $s=s_0$. Let $(a,b)\in\Omega$ such that $h_{s_{0}}g\mathcal{L}=M_{a,b}\mathcal{L}$ or $h_{s_{0}}g\mathcal{L}=M_{a,b}\mathcal{L}^{R}$. Note that the return time for the $T$ orbit of $M_{a,b}\mathcal{L}$ is given by $R(T^{i}(a,b))=s_{i+1}-s_{i}$ for $i=0,\dotsc,N,\dotsc$.  

Therefore the corresponding set of first $N$ gaps in the vertical strip for $g\mathcal{L}$ can be written as 
\[
\overline{\mathbb{G}}_{g\mathcal{L}}^{N}=\{s_{i+1}-s_i\mid i=0,\dotsc,N-1\}=\{R(T^{i}(a,b))\mid i=0,1,\dotsc,N-1\}.
\]

Thus, the question of understanding the gaps in the vertical strip has been translated into understanding the return times to the Poincar\'e section and we can express the proportion of gaps that are bounded below by some positive $t>0$ as a Birkhoff sum of the indicator function $\chi_{R^{-1}([t,\infty))}$: 
\[
\frac{1}{N}\left|\overline{\mathbb{G}}_{g\mathcal{L}}^{N}\cap[t,\infty)\right|=\frac{1}{N}\sum_{i=0}^{N-1}\chi_{R^{-1}([t,\infty))}(T^i(a,b)).
\]

Following the work of Athreya, Chaika, and Leli\`evre in \cite{ACL}, we can determine the ergodic invariant measures for the return map $T$ by using its interpretation as the first return map for horocycle flow on the space $M = SL(2,\RR)/\Gamma$. Results due to Sarnak \cite{S81} and Dani-Smillie \cite{DS84} on the equidistribution of long periodic orbits of $h_s$ allow us to state similar equidistribution results for $T$.
The following two theorems are adaptations of \cite[Theorem 4.1, Theorem 4.2 ]{ACL}. 

\begin{thm}\label{4.1}\cite{ACL} Other than measures supported on periodic orbits, the Lebesgue probability measure $m$ given by $dm = \frac{2}{3+\sqrt{2}}\, dadb$ is the unique ergodic invariant probability measure for $T$. For every $(a,b)$ which is not $T$-periodic and any continuous, compactly supported function $f \in C_c(\Omega)$, we have that
\[\lim_{N\rightarrow\infty}\frac{1}{N}\sum^{N-1}_{i=0}f(T^i(a,b))=\int_\Omega f\, dm\]
Moreover, if $\{(a_r,b_r)\}^\infty_{r=1}$ is a sequence of periodic points with periods $N(r) \rightarrow \infty$ as $r \rightarrow \infty$, we have, for any bounded function $f$ on $\Omega$,
\[\lim_{r\rightarrow\infty}\frac{1}{N(r)}\sum^{N(r)-1}_{i=0}f(T^i(a_r,b_r))=\int_\Omega f \, dm.\]
\end{thm}

Define $a_r:=\left[\begin{array}{cc}e^{r/2} & 0 \\ 0 & e^{-r/2}\end{array}\right]$ whose action on the moduli space is known as the \emph{geodesic flow}. Together with the observation in Section \ref{verticalstrip}, Theorem \ref{4.1} implies the following more general statement:

\begin{thm}\label{compthm}\cite{ACL}
Consider $\gamma = g\Gamma \in M=SL(2,\RR)/\Gamma$, which is not $h_s$-periodic. Then for $t \geq 0$,
\[\lim_{N\rightarrow\infty}\frac{1}{N}\left|\overline{\mathbb{G}}_\gamma^N \cap [t,\infty)\right| = m(\{(a,b) \in \Omega : R(a,b) \geq t\}).\]
If $\gamma$ is $h_s$-periodic, then we define $\gamma_r=a_{-r}g\Gamma = a_{-r}\gamma$. Then there is an $M(r)$ such that for $N \geq M(r)$, we have $\overline{\mathbb{G}}_{\gamma_r}^N = \overline{\mathbb{G}}_{\gamma_r}^{M(r)}$ and
\[\lim_{r\rightarrow\infty}\frac{1}{M(r)}\left|\overline{\mathbb{G}}_{\gamma_r}^{M(r)}\cap [t,\infty)\right| = m(\{(a,b)\in \Omega : R(a,b) \geq t\}).\]
\end{thm}

\noindent In particular this implies that for $\mathcal{L}$ and $g\mathcal{L}$ where $g \in SL(2,\RR)$, the limiting gap distributions will be the same.


The next lemma establishes a relation between the renormalized slope gaps in $\mathcal{L}$ in the $\ell_{\infty}$ ball of radius $R$ around the origin and the slope gaps in the vertical strip for an appropriate element in the $SL(2,\RR)$ orbit of $\mathcal{L}$.  

\begin{lem}\label{balltostrip} Let $g_{R}\mathcal{L}=a_{-2\log{R}}\mathcal{L}=\left[\begin{array}{cc}R^{-1} & 0 \\ 0 & R\end{array}\right]\mathcal{L}$. Then,
\[
\frac{1}{N}\left|\mathbb{G}_\mathcal{L}^{R}\cap[t,\infty)\right|=\frac{1}{N}\left|\overline{\mathbb{G}}_{g_{R}\mathcal{L}}^{N(R)}\cap[t,\infty)\right|.
\]
\end{lem}
\begin{proof} If we look at the set of saddle connections on $\mathcal{L}$ that lie in the $\ell_{\infty}$ ball of radius $R$ about the origin, each of them has horizontal component $\le R$. Hence, after applying the matrix $g_{R}=\left[\begin{array}{cc}R^{-1} & 0 \\ 0 & R\end{array}\right]$, each of these saddle connections will be mapped to a saddle connection on $g_{R}\mathcal{L}$ in the vertical strip.  The corresponding slopes, hence the gaps, will be scaled by $R^{2}$. Therefore, the gap set of first $N(R)$ saddle connections on $g_{R}\mathcal{L}$ in the vertical strip is precisely the set of renormalized gaps for $\mathcal{L}$ in the open ball of radius $R$. 
\end{proof}

\subsection{Proof of Thm \ref{octthm}}\label{11pf}

We recall the statement

\begin{namedtheorem}[Theorem \ref{octthm}]\label{Lgaps}There is a limiting probability distribution function $f:[0,\infty)\to[0,1]$ such that 
\[
\lim_{R\to\infty}\frac{|\mathbb{G}_\mathcal{L}^{R}\cap(a,b)|}{N(R)}=\int_{a}^{b}f(x)dx.
\]
The function $f$ is continuous, piecewise differentiable with seven points of non-differentiability. Each piece is expressed in terms of elementary functions. 
\end{namedtheorem}

\begin{figure}[h!]
\begin{minipage}{.49\textwidth}
\centering
\includegraphics[width=2.5in]{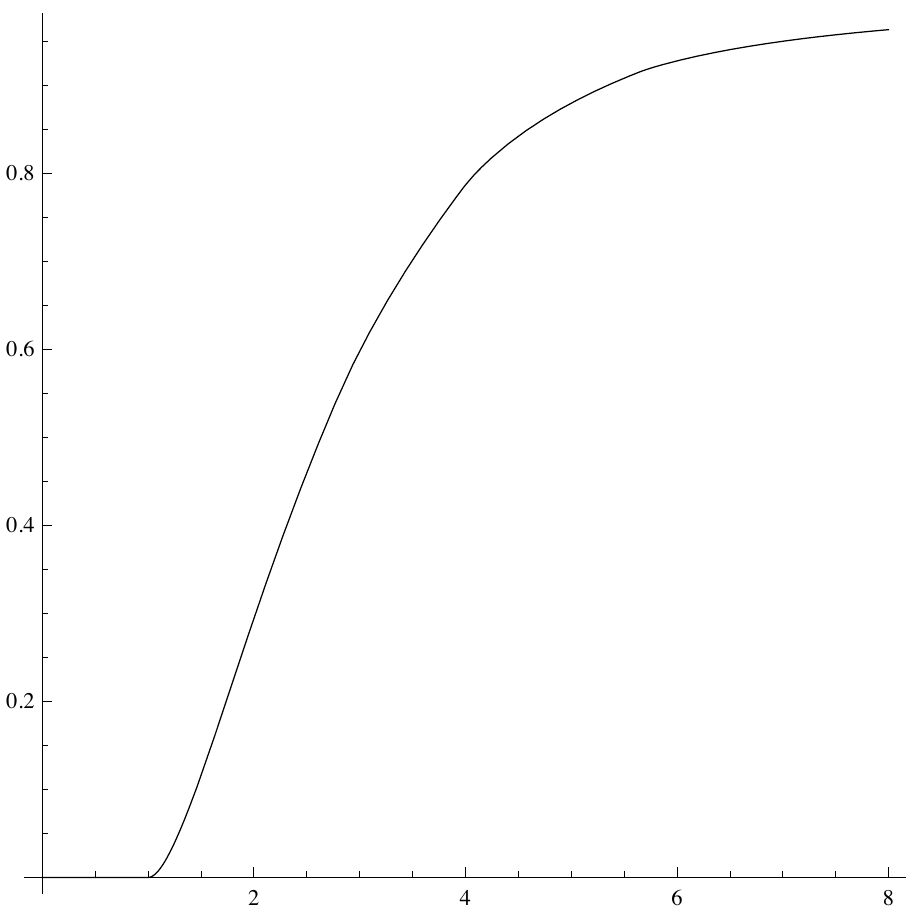}
\caption{Cumulative distribution function}
\label{CDF}
\end{minipage}
\begin{minipage}{.49\textwidth}
\centering
\includegraphics[width=2.5in]{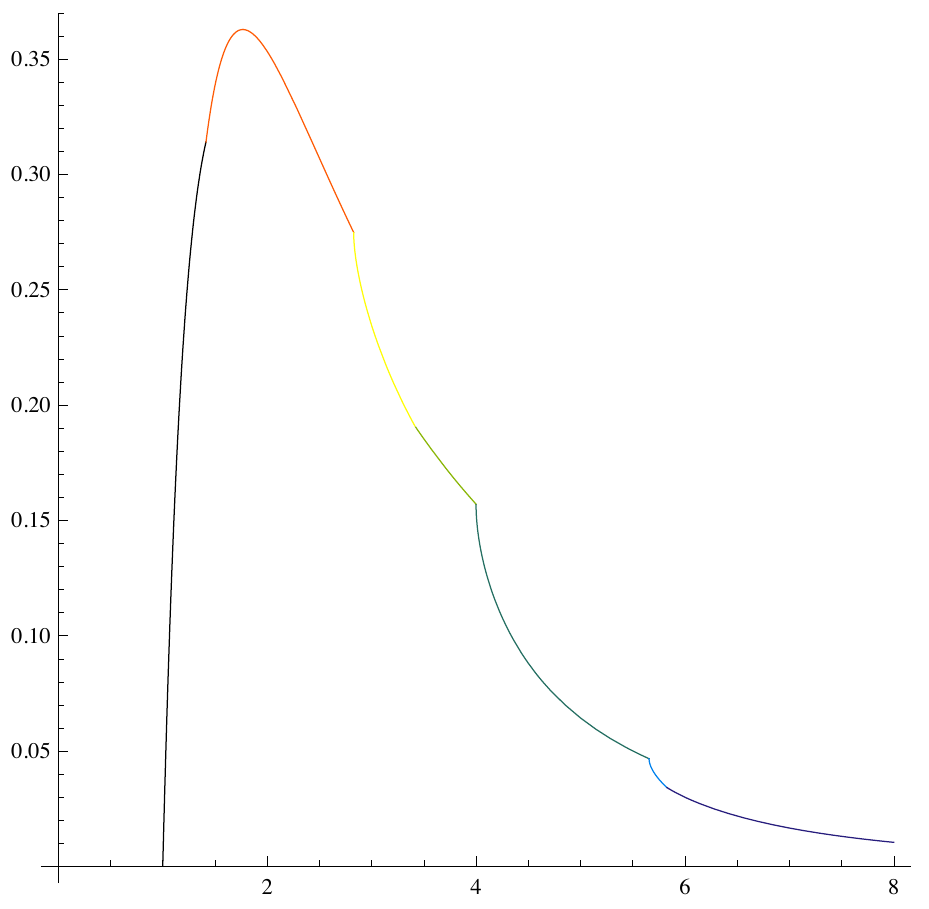}
\caption{Limiting gap distribution for $\mathcal{L}$}
\label{dist}
\end{minipage}
\end{figure}

\begin{proof} The translation surface associated to the normalized $L$ is $h_s$ periodic, since the matrix $h_{-\sqrt{2}-1}=\left[\begin{array}{cc}1 & 0 \\ \sqrt{2}+1 & 1\end{array}\right]$ is in the Veech group $\Gamma$. Now, combining Lemma \ref{balltostrip} and Theorem \ref{compthm}, we get
\[
\lim_{R\to\infty}\frac{1}{N}\left|\mathbb{G}_\mathcal{L}^{R}\cap[t,\infty)\right|=\lim_{R\to\infty}\frac{1}{N}\left|\overline{\mathbb{G}}_{g_{R}\mathcal{L}}^{N(R)}\cap[t,\infty)\right|=m(\{(a,b)\in \Omega : R(a,b) \geq t\}).
\]

To find the cumulative distribution function (cdf) we compute the area under the return time function. This construction of the cdf depends on the intersection of the return time hyperbolas with the subregions of $\Omega_1\cup\Omega_2$. On the $\Omega_{AH}$ region, the integral changes shape at two points, when $R(a,b) = 1$, corresponding to when the hyperbola enters the region, and when $R(a,b) = 4$, corresponding to when the hyperbola becomes tangent to the line $b=1-a$. On $\Omega_{BD}$, the hyperbola enters the region when $R(a,b) = 1$, becomes tangent to the lines $b = \frac{1-(1+\sqrt{2}a)}{\sqrt{2}}$ and $b = 1-(1+\sqrt{2}a)$ when $R(a,b) =4\sqrt{2}$ and when $R(a,b) =\frac{1}{(\sqrt{2}-1)^2}$ it intersects the point $(\sqrt{2}-1,0)$, thus the region of integration changes shape 3 times. For $\Omega_{HF}$, the hyperbola enters the region when $R(a,b) = \sqrt{2}$, becomes tangent to the line $b=1-(1+\sqrt{2}a)$ at $R(a,b) =2\sqrt{2}$ and exits the region when $R(a,b) =2+\sqrt{2}$. On the $\Omega_{AE}$ region, the hyperbola enters the region when $R(a,b) =1$ and becomes tangent to the line $b=1-(\sqrt{2}+1)a$ when $R(a,b) =4$. On $\Omega_2$, our function behaves the same as on $\Omega_{AH}$. Thus our cumulative distribution function, see Figure \ref{CDF}, is a piecewise function defined on 7 subintervals from $[1,\infty)$.

Differentiating the cumulative distribution function yields a graph with 7 points of non differentiability describing the distribution of the gaps between the slopes of saddle connections on the octagon, see Figure \ref{dist}.

\begin{align*}
f_1(t) &= -\frac{8\ln{\frac{1}{t}}}{(3+\sqrt{2})t^2}, \text{ for } t \in \left[1, \sqrt{2}\right)\\
f_2(t) &= -\frac{\ln{2}+10\ln{\frac{1}{t}}}{(3+\sqrt{2})t^2}, \text{ for } t \in \left[\sqrt{2},2\sqrt{2}\right)\\
f_3(t) &= -\frac{4\tanh^{-1}\left(\frac{\sqrt{t(-2\sqrt{2}+t)}}{t}\right) + \ln{2} + 10\ln{\frac{1}{t}}}{(3+\sqrt{2})t^2}, \text{ for } t \in \left[2\sqrt{2},\sqrt{2}+2\right)\\
f_4(t) & = -\frac{8\ln{\frac{1}{t}}}{(3+\sqrt{2})t^2}, \text{ for } t \in \left[\sqrt{2}+2, 4\right)\\
f_5(t) &= -\frac{12\tanh^{-1}\left(\sqrt{\frac{-4+t}{t}}\right)+8\ln{\frac{1}{t}}}{(3+\sqrt{2})t^2}, \text{ for } t \in \left[4, 4\sqrt{2}\right)\\
f_6(t) &= -\frac{4\left(3\tanh^{-1}\left(\sqrt{\frac{-4+t}{t}}\right)+2\left(\tanh^{-1}\left(\frac{\sqrt{t(-4\sqrt{2}+t)}}{t}\right)+\ln{\frac{1}{t}}\right)\right)}{(3+\sqrt{2})t^2}, \text{ for } t \in \left[4\sqrt{2},\frac{1}{(\sqrt{2}-1)^2}\right)\\
f_7(t) &= -\frac{24\tanh^{-1}\left(\sqrt{\frac{-4+t}{t}}\right) + \ln{64} + 16\ln{\frac{1}{t}} + 4\ln{\frac{t+\sqrt{t(-4\sqrt{2}+t)}}{2t}-4\ln\left(1-\frac{\sqrt{t(-4\sqrt{2}+t)}}{t}\right)}}{2(3+\sqrt{2})t^2},\\&\hspace{324pt} \text{ for } t \in \left[\frac{1}{(\sqrt{2}-1)^2},\infty\right)
\end{align*}
\end{proof}

\subsection{Volume Computations} \label{volume} Since $\Omega$ is a Poincar\'e section, the integral of the return time should yield the covolume of the Veech group. Recall the definition of the dilogarithm function
\[\Li_2(z) = \sum_{k=1}^\infty \frac{z^k}{k^2} =  \int_z^0 \frac{\ln(1-t)\ dt}{t}\]

Integral over AH region: 
\begin{displaymath}
\int_0^1\int_{1-x}^1 \frac{1}{xy} dy dx = \frac{\pi^2}{6}
\end{displaymath}

Integral over BD region:
\[
\int_{\sqrt{2}-1}^1 \int_{\frac{1-(1+\sqrt{2})x}{\sqrt{2}}}^{1-x} \frac{1}{x(x+y)} \ dydx = \Li_{2}(1)-\Li_2({\sqrt{2}-1})-\ln{(\sqrt{2})}\ln{(\sqrt{2}-1)}.
\]

\[
\int_{0}^{\sqrt{2}-1} \int_{(-1-\sqrt{2})x+1}^{1-x} \frac{1}{x(x+y)}\ dydx= \Li_{2}(2-\sqrt{2}).
\]
\\
Integral over HF region:
\[
\int_{\sqrt{2}-1}^1 \int_{1-(1+\sqrt{2})x}^{\frac{1-(1+\sqrt{2})x}{\sqrt{2}}} \frac{\sqrt{2}}{x((1+\sqrt{2})x+\sqrt{2}y)}\ dydx=\Li_2\left(\frac{1}{\sqrt{2}}\right)-\Li_2\left(1-\frac{1}{\sqrt{2}}\right)+\ln{(\sqrt{2})}\ln{(\sqrt{2}-1)}.
\]
\\
Integral over AE region:
\[
\int_{\sqrt{2}-1}^{1} \int_{-(2+\sqrt{2})x+1}^{1-(1+\sqrt{2})x} \frac{1}{x((1+\sqrt{2})x+y)}\ dydx=\frac{\pi^2}{6}.
\]
Integrating over $\Omega_2$ we get
\begin{displaymath}
\int_0^1\int_{1-x}^1 \frac{1}{xy} dy dx = \frac{\pi^2}{6}.
\end{displaymath}
Therefore sum of the integrals: \[
\int\int_{\Omega}R(x,y)\ dydx= \]
\[
\frac{4\pi^2}{6}-\Li_2(\sqrt{2}-1)-\ln{(\sqrt{2})}\ln{(\sqrt{2}-1)}+\Li_2(2-\sqrt{2})+\Li_2\left(\frac{1}{\sqrt{2}}\right)-\Li_2\left(1-\frac{1}{\sqrt{2}}\right)+\ln{(\sqrt{2})}\ln{(\sqrt{2}-1)}.
\]

By using the following dilogarithm identities, \cite{Bridgeman}

\begin{equation} \label{eq:dilog1}
\Li_2(z)=-\Li_2(1-z)-\ln{(1-z)}\ln{(z)}+\frac{\pi^2}{6}
\end{equation}

\begin{equation} \label{eq:dilog2}
\Li_2(z)+\Li_2(-z)=\frac{1}{2}\Li_2(z^2)
\end{equation}

\begin{equation} \label{eq:dilog3}
\Li_2(z)=-\Li_2\left(\frac{1}{z}\right)-\frac{1}{2}\ln^2{(-z)}-\frac{\pi^2}{6}\ \text{ for } z\notin(0,1)
\end{equation}

\begin{equation} \label{eq:dilog4}
\Li_2(z)=-\Li_2\left(\frac{z}{z-1}\right)-\frac{1}{2}\ln^2{(1-z)}\ \text{ for } z\notin(1,\infty)
\end{equation}

we can deduce that \[
\int\int_{\Omega}R(x,y)\ dydx=\frac{\pi^2}{4}-\frac{1}{4}\ln^2{(2-\sqrt{2})}-\frac{1}{2}\ln^2{(\sqrt{2}-1)}+\ln{\left(\frac{1}{\sqrt{2}}\right)}\ln{\left(1-\frac{1}{\sqrt{2}}\right)}+\frac{\pi i\ln{2}}{2}-\frac{1}{2}\ln^{2}{(-\sqrt{2})}
\]
which can be further reduced to
\[
\int\int_{\Omega}R(x,y)\ dydx=\frac{3\pi^2}{4}
\] 
by using logarithm identities, which is precisely the covolume of the Veech group $\Gamma$. Therefore $\Omega$ is indeed a full section for the horoccyle flow.


\section{Generalization to higher number of cusps}\label{poincaregeneral}

Let  $(X,\omega)$ be a Veech surface with $n < \infty$ cusps. In this section, we describe a parametrization of a Poincar\'e section for the horocycle flow on $SL(2,\RR)/SL(X,\omega)$ generalizing our strategy in Section \ref{poincare}. We adhere to the notation introduced in Section \ref{poincare}.

Let $\Gamma_1,\dotsc,\Gamma_n$ be maximal parabolic sugroups of $SL(X,\omega)$ representing the conjugacy classes of all maximal parabolic subgroups. For any (and hence all) $i=1,\dotsc,n$, the subgroup $\Gamma_i$ is isomorphic to either $\mathbb{Z}\oplus\mathbb{Z}/2\mathbb{Z}$ or $\mathbb{Z}$, depending on whether $-Id=\left[\begin{array}{cc}-1 & 0 \\ 0 & -1\end{array}\right]$ belongs to the Veech group $SL(X,\omega)$ or not.  Let $P_i\in\Gamma_i$ be a generator for the infinite cyclic factor. Note that, in the first case where $\Gamma_i\cong\mathbb{Z}\oplus\mathbb{Z}/2\mathbb{Z}$, the element $P_i$ can always be chosen to have eigenvalue $1$.   

We can decompose the set $\Lambda_{sc}(X,\omega)$  into a finite union of disjoint $SL(X,\omega)$ orbits of saddle connections. That is, there exist $v_1, \ldots, v_m \in \RR^2$ such that 

\[
\Lambda_{sc}(X,\omega)=\bigcup_{j=1}^{m} SL(X,\omega)v_j.
\]
In this decomposition, every $v_j$ is parallel to a direction determined by a parabolic subgroup $\Gamma_{i}$ , i.e. it is in the unique non-trivial eigenspace of the parabolic element $P_i$, see \cite{V89}. In each direction we can then find the shortest vector, $\pm w_i$. We denote the Veech group orbits of each of these vectors by $\Lambda_{sc}^{\pm w_i}(X,\omega) := SL(X,\omega)\cdot w_i\cup SL(X,\omega)\cdot -w_i$.

First assume that the Veech group $SL(X,\omega)$ contains the element $-Id$. 
Note that, in this case $\Lambda_{sc}^{\pm w_i}(X,\omega)=\Lambda^{w_i}_{sc}(X,\omega):=SL(X,\omega)\cdot w_i$. 
As we noted above, in this case the parabolic generator $P_i$ of the infinite cyclic factor of the maximal parabolic subgroup $\Gamma_i$ can be chosen to have eigenvalue 1. Thus for the parabolic element $P_i$ associated to  $w_i$ there exists $C_i\in SL(2,\RR)$ such that
\[S_i=C_iP_iC_i^{-1} = \left[\begin{array}{cc}1 & \alpha_i \\ 0 & 1\end{array}\right].\]
After replacing $P_i$ with its inverse if necessary, we can assume that  $\alpha_i>0$. 
Since $w_i$ is an eigenvector for the matrix $P_i$, the vector $C_i\cdot w_i$ is an eigenvector for \[\left[\begin{array}{cc}1 & \alpha_i \\ 0 & 1\end{array}\right],\] hence it can be written as \[C_i\cdot w_i=\left[\begin{array}{c}\beta_i\\0\end{array}\right]\] for some $\beta_i\neq0$. Moreover, we can choose $C_i$ in a way that $\beta_i=1$ for all $i$. So we will assume that \[C_i\cdot w_i= \left[\begin{array}{c}1\\0\end{array}\right].\]

Recall that in \cite{Ath13}, Athreya proved that the set \[
\Omega^{M}=\{gSL(X,\omega)\mid g(X,\omega)\text{ has a horizontal saddle connection of length $\le1$}\}\]
is a Poincar\'e section for the action of  horocycle flow on the moduli space $SL(2,\RR)/SL(X,\omega)$. 
By using the above decomposition of the set of saddle connections on $(X,\omega)$  we will write
\[
\Omega^{M} = \Omega_1^{M} \cup \cdots \cup \Omega_n^{M},\]
where
\[\Omega_i^{M} = \{gSL(X,\omega)\mid g\Lambda_{sc}^{\pm w_i}(X,\omega)\cap(0,1]\neq\emptyset\}.
\]

\noindent For each $i\in\{1,\dotsc,n\}$ we define the following subset of $\RR^2$: 
\[\Omega_i=\{(a,b)\in\mathbb{R}^{2}\mid 0<a\le1\ \text{and}\  1-(\alpha_i)a<b\le1\}.\]

Now, assume that the Veech group $SL(X,\omega)$ does not contain $-Id$. Hence for each $i$, $\Gamma_i\cong\mathbb{Z}$, and for some $i$, the generator $P_i$ can possibly have eigenvalue $-1$. Now, possibly after a renaming, assume that for $i=1,\dotsc,k$, the generator $P_i$ has eigenvalue $1$, and for $i=k+1,\dotsc,n$, the generator $P_i$ has eigenvalue $-1$, where $1\le k\le n$. 

Then, by conjugating each $P_i$ with an element $C_i\in SL(2,\RR)$ we obtain

\[S_i=C_iP_iC_i^{-1} = \left[\begin{array}{cc}1 & \alpha_i \\ 0 & 1\end{array}\right]\]
for $i=1,\dotsc,k$, and, 
\[
S_i=C_iP_iC_i^{-1} = \left[\begin{array}{cc}-1 & \alpha_i \\ 0 & -1\end{array}\right]\] for $i=k+1,\dotsc,n$. 
Again,  we will assume that each $\alpha_i>0$. Similar to the first case, we have \[C_i\cdot w_i=\left[\begin{array}{c}1\\0\end{array}\right].\]
Now for each $i=1,\dotsc,k$ define $\Omega_i$ as follows: 

\[\Omega_i:=\{(a,b)\in\mathbb{R}^{2}\mid 0<a\le1\ \text{and}\  1-(\alpha_i)a<b\le1\}\cup\{(a,b)\in\mathbb{R}^{2}\mid -1\le a<0\ \text{and}\  1+(\alpha_i)a<b\le1\}.\]

\noindent For $i=k+1,\dotsc,n$, define $\Omega_i$ as follows:
\[
\Omega_i:=\{(a,b)\in\mathbb{R}^{2}\mid 0<a\le1, \  1-2(\alpha_i)a<b\le1\}\]

\begin{namedtheorem}[Theorem \ref{parametrizingsection}]\label{bijection}
Let $(X,\omega)$ be a Veech surface such that the Veech group $SL(X,\omega)$ has $n<\infty$ cusps. There are coordinates from the section $\Omega^M$ to the set $\bigsqcup_{i=1}^{n}\Omega_i$. More precisely, the bijection between $SL(2,\RR)/SL(X,\omega)$ and $SL(2,\RR)\cdot (X,\omega)$ sends $\Omega_{i}^M$ to $\{M_{a,b}C_{i}(X,\omega)\mid(a,b)\in\Omega_i\}$, and the latter set is bijectively parametrized by $\Omega_i$. The return time function is piecewise rational with pieces defined by linear equations in these coordinates. The limiting gap distribution for any Veech surface is piecewise real analytic.
\end{namedtheorem}

\begin{proof}We divide the proof into two cases. \\
\noindent\textbf{Case 1.} $SL(X,\omega)$ contains $-Id$.\\
 
Suppose that $gSL(X,\omega)\in \Omega_i^M$. We will first show that $g(X,\omega)=M_{a,b}C_i(X,\omega)$ for some $(a,b)\in\Omega_i$. By definition of the $\Omega_{i}^{M}$, $g(X,\omega)$ has a horizontal saddle connection 
\[
\left[\begin{array}{c}a\\0\end{array}\right],\text{ with }0< a\leq1,
\] such that  
\[
\left[\begin{array}{c}a\\0\end{array}\right] = g\gamma (\pm w_i)
\] for some $\gamma \in SL(X,\omega)$. Since $-Id\in SL(X,\omega)$, the vector $-w_i$ is in $SL(X,\omega)$ orbit of the vector $w_i$, and $\Lambda_{sc}^{\pm w_i}(X,\omega)=\Lambda_{sc}^{w_i}(X,\omega)$.  Hence we can assume that  
\[
\left[\begin{array}{c}a\\0\end{array}\right] = g\gamma w_i.
\]
Since $SL(X,\omega)$ acts on $\Lambda_{sc}^{w_i}(X,\omega)$ transitively and for any element $\gamma \in SL(X,\omega)$ we have $g\gamma (X,\omega) = g(X,\omega)$, we can assume that \[
\left[\begin{array}{c}a\\0\end{array}\right] = gw_i=gC^{-1}_i\left[\begin{array}{c}1\\0\end{array}\right].
\]
The elements in $SL(2, \RR)$ which take $\left[\begin{array}{c}1\\0\end{array}\right]$ to $\left[\begin{array}{c}a\\0\end{array}\right]$ are precisely of the form
\[
M_{a,b}=\left[\begin{array}{cc}a & b \\ 0 & 1/a\end{array}\right],\]
 where $b \in \RR$. Therefore, $gC_i^{-1}=M_{a,b}$ and hence $g(X,\omega)=M_{a,b} C_i(X,\omega)$ with $0<a\le1$. 

Since the parabolic element $S_{i} = \left[\begin{array}{cc}1 & \alpha_i \\ 0 & 1\end{array}\right]$ is in the Veech group of $C_i(X,\omega)$, and \[
M_{a,b}S_i^{-n} = \left[\begin{array}{cc}a & b-(\alpha_i)an \\ 0 & a^{-1}\end{array}\right],
\]
we can write $g(X,\omega)=M_{a,b}C_i(X,\omega)$, where $(a,b) \in \Omega_i$, as required.

In order to see that for every $(a,b)\in\Omega_i$, the element $M_{a,b}C_i(X,\omega)$ lies in $\Theta(\Omega_i^{M})$, we first note that \[
C_i\cdot w_i=\left[\begin{array}{c}1\\0\end{array}\right]
\] is a horizontal saddle connection in $C_i(X,\omega)$ with length at most $1$. Since $M_{a,b}$ sends this to a saddle connection of length $a\leq1$, it follows that $M_{a,b}C_i(X,\omega)\in\Theta(\Omega_i^{M})$.

 

Finally, we prove that $\Omega_i$ bijectively parametrizes $\{M_{a,b}C_i(X,\omega)\mid (a,b)\in\Omega_i\}$. We suppose that $(a,b),(c,d)\in\Omega_i$ and that $M_{a,b}C_i(X,\omega)=M_{c,d}C_i(X,\omega)$, and must prove that $(a,b)=(c,d)$. We have \[
M_{a,b}M_{c,d}^{-1} = \left[\begin{array}{cc}\frac{a}{c} & bc-ad \\ 0 & \frac{c}{a}\end{array}\right] \in C_iSL(X,\omega)C^{-1}_i,
\] which fixes infinity in the upper half-plane. Since $S_i$ generates the infinite cyclic factor in $C_i\Gamma_iC^{-1}_i$ and $a/c>0$ it follows that $a=c$. From this we see that $1-(\alpha_i)a < b,d \leq 1$, and $(b-d)a = k \alpha_i$, which implies $b = d$, and hence $(a,b) = (c,d)$, as required.\\

\noindent\textbf{Case 2.} $-Id\notin SL(X,\omega)$. \\
\noindent\textbf{Subcase 2.1.}
Suppose that $gSL(X,\omega)\in\Omega_{i}^{M}$ for some $i=1,\dotsc,k$.

By definition of the set $\Omega_i^{M}$, $g(X,\omega)$ has a saddle connection \[\left[\begin{array}{c}a\\0\end{array}\right],\ \text{with}\ 0<a\le1,\] such that either  \[\left[\begin{array}{c}a\\0\end{array}\right]=g\gamma w_i\ \text{or}\ \left[\begin{array}{c}a\\0\end{array}\right]=g\gamma (-w_i).\] Since $SL(X,\omega)$ acts transitively on both $\Lambda_{sc}^{w_i}(X,\omega)$ and $\Lambda_{sc}^{-w_i}(X,\omega)$, as in the first case we can assume that either \[\left[\begin{array}{c}a\\0\end{array}\right]=gw_i\ \text{or}\ \left[\begin{array}{c}a\\0\end{array}\right]=g(-w_i).
\] 
This means that either \[\left[\begin{array}{c}a\\0\end{array}\right]=gC_i^{-1}\left[\begin{array}{c}1\\0\end{array}\right]\ \text{or}\ \left[\begin{array}{c}a\\0\end{array}\right]=gC_i^{-1}\left[\begin{array}{c}-1\\0\end{array}\right].
\]
The set of matrices that take $\left[\begin{array}{c}1\\0\end{array}\right]$ to $\left[\begin{array}{c}a\\0\end{array}\right]$ are of the form \[
M_{a,b}=\left[\begin{array}{cc}a& b \\ 0 & 1/a\end{array}\right],\]
 for some $b$. Similarly, the set of matrices that take $\left[\begin{array}{c}-1\\0\end{array}\right]$ to $\left[\begin{array}{c}a\\0\end{array}\right]$ are of the form 
\[
M_{-a,b}=\left[\begin{array}{cc}-a & b \\ 0 & -1/a\end{array}\right],\]
for some $b$.  Therefore we have either $gC_i^{-1}=M_{a,b}$ or $gC_i^{-1}=M_{-a,b}$, and hence either $g(X,\omega)=M_{a,b}C_i(X,\omega)$ or $g(X,\omega)=M_{-a,b}C_i(X,\omega)$ with $0<a\le1$, where the latter can be rewritten as $g(X,\omega)=M_{a,b}C_i(X,\omega)$ with $-1\leq a<0$. 

Since the parabolic element $S_{i} = \left[\begin{array}{cc}1 & \alpha_i \\ 0 & 1\end{array}\right]$ is in the Veech group of $C_i(X,\omega)$, and \[
M_{a,b}S_i^{\pm n} = \left[\begin{array}{cc}a & b\pm(\alpha_i)an \\ 0 & a^{-1}\end{array}\right],
\]
we can write $g(X,\omega)=M_{a,b}C_i(X,\omega)$ where $(a,b) \in \Omega_i$, as required. 

In order to see that for every $(a,b)\in\Omega_i$, the element $M_{a,b}C_i(X,\omega)$ lies in $\Theta(\Omega_i^{M})$, we first note that \[
C_i(\pm w_i)=\left[\begin{array}{c}\pm1\\0\end{array}\right]
\]
is a horizontal saddle connection in $C_i(X,\omega)$ with length at most $1$. Since $M_{a,b}$ sends this to a saddle connection of length $a\leq1$, it follows that $M_{a,b}C_i(X,\omega)\in\Theta(\Omega_i^{M})$.

Finally, we prove that $\Omega_i$ bijectively parametrizes $\{M_{a,b}C_i(X,\omega)\mid (a,b)\in\Omega_i\}$. We suppose that $(a,b),(c,d)\in\Omega_i$ and that $M_{a,b}C_i(X,\omega)=M_{c,d}C_i(X,\omega)$, and must prove that $(a,b)=(c,d)$. We have 
\[
M_{a,b}M_{c,d}^{-1} = \left[\begin{array}{cc}\frac{a}{c} & bc-ad \\ 0 & \frac{c}{a}\end{array}\right] \in C_iSL(X,\omega)C^{-1}_i,
\]
which fixes infinity in the upper half-plane. Since $S_i$ generates the the maximal parabolic subgroup $C_i\Gamma_iC^{-1}_i$ it follows that $a=c$. From this we see that $1-(\alpha_i)a < b,d \leq 1$, and $(b-d)a = k \alpha_i$, which implies $b = d$, and hence $(a,b) = (c,d)$, as required.\\

\noindent \textbf{Subcase 2.2.} Suppose that $gSL(X,\omega)\in\Omega_i^{M}$ for some $i=k+1,\dotsc,n$.

An identical argument as above says that we can write $g(X,\omega)=M_{a,b}C_i(X,\omega)$ where either $0<a<1$ or $-1\le a<0$. 

Observe that \[
M_{a,b}\left[\begin{array}{cc}-1 & \alpha_i \\ 0 & -1\end{array}\right]=\left[\begin{array}{cc}a & b \\ 0 & 1/a\end{array}\right]\left[\begin{array}{cc}-1 & \alpha_i \\ 0 & -1\end{array}\right]=\left[\begin{array}{cc}-a & a\alpha_i-b \\ 0 & -1/a\end{array}\right]=M_{-a,a\alpha_i-b}.
\]
Hence, for any point $(a,b)$ where $-1\leq a<0$, there is a point $(a',b')=(-a,a\alpha_i-b)$ with  $0<a'\le1$ such that $M_{a,b}C_i(X,\omega)=M_{a',b'}C_i(X,\omega)$. Moreover, the element  \[
S_i^{-2}=\left[\begin{array}{cc}1 & 2\alpha_i \\ 0 & 1\end{array}\right]\in C_iSL(X,\omega)C^{-1}_i,\]
thus for $gSL(X,\omega)\in\Omega_i^{M}$ we have $g(X,\omega)=M_{a,b}C_i(X,\omega)$ for some $(a,b)\in\Omega_i$. 

Finally, we prove that $\Omega_i$ bijectively parametrizes $\{M_{a,b}C_i(X,\omega)\mid (a,b)\in\Omega_i\}$. We suppose that $(a,b),(c,d)\in\Omega_i$ and that $M_{a,b}C_i(X,\omega)=M_{c,d}C_i(X,\omega)$, and must prove that $(a,b)=(c,d)$. We have  \[
M_{a,b} \cdot M_{c,d}^{-1} = \left[\begin{array}{cc}\frac{a}{c} & bc-ad \\ 0 & \frac{c}{a}\end{array}\right] \in C_iSL(X,\omega)C^{-1}_i,\]
 which fixes infinity in the upper half-plane. Since $S_i$ generates the the maximal parabolic subgroup $C_i\Gamma_iC^{-1}_i$, and $a/c>0$ it follows that $a=c$. From this we see that $1-2(\alpha_i)a < b,d \leq 1$, and $(b-d)a = 2k \alpha_i$, which implies $b = d$, and hence $(a,b) = (c,d)$, as required.\\
 
As we observed in the Proof of Theorem \ref{Lgaps}, in these coordinates, the return time function at a point $M_{a,b}C_i(X,\omega)$ is given by the smallest slope of a saddle connection in the vertical strip. Hence it is of the form 
\[
\frac{y}{a(ax+by)},
\]
where $\left[\begin{array}{c}x\\y\end{array}\right]$ is a saddle connection on $C_i(X,\omega)$. Note that the return time function is composed of rational pieces defined on convex polygons. Since the cumulative distribution function is given by the total area bounded by the return time hyperbolas and these convex polygons, it is piecewise real analytic. Hence, the limiting gap distribution function for any Veech surface is real analytic, which finishes the proof. 

\end{proof}

\bibliographystyle{abbrv}
\bibliography{gapsreferences}

\begin{thebibliography}{10}

\bibitem{Ath13}
J.~S. Athreya.
\newblock Gap distributions and homogeneous dynamics.
\newblock to appear in Proceedings of ICM Satellite Conference on Geometry,
  Topology, and Dyanamics in Negative Curvature.

\bibitem{ACL}
J.~S. Athreya, J.~Chaika, and S.~Leli{\`e}vre.
\newblock The gap distribution of slopes on the golden {L}.
\newblock In {\em Recent trends in ergodic theory and dynamical systems},
  volume 631 of {\em Contemp. Math.}, pages 47--62. Amer. Math. Soc.,
  Providence, RI, 2015.

\bibitem{AC13}
J.~S. Athreya and Y.~Cheung.
\newblock A {P}oincar\'e section for the horocycle flow on the space of
  lattices.
\newblock {\em Int. Math. Res. Not. IMRN}, (10):2643--2690, 2014.

\bibitem{Bridgeman}
M.~Bridgeman.
\newblock Orthospectra of geodesic laminations and dilogarithm identities on
  moduli space.
\newblock {\em Geom. Topol.}, 15(2):707--733, 2011.

\bibitem{DS84}
S.~G. Dani and J.~Smillie.
\newblock Uniform distribution of horocycle orbits for {F}uchsian groups.
\newblock {\em Duke Mathematical Journal}, 51(1):184--194, 1984.

\bibitem{HS}
P.~Hubert and T.~A. Schmidt.
\newblock An introduction to {V}eech surfaces.
\newblock In {\em Handbook of dynamical systems. {V}ol. 1{B}}, pages 501--526.
  Elsevier B. V., Amsterdam, 2006.

\bibitem{Masur}
H.~Masur.
\newblock Ergodic theory of translation surfaces.
\newblock In {\em Handbook of dynamical systems. {V}ol. 1{B}}, pages 527--547.
  Elsevier B. V., Amsterdam, 2006.

\bibitem{MT}
H.~Masur and S.~Tabachnikov.
\newblock Rational billiards and flat structures.
\newblock {\em Handbook of dynamical systems}, 1:1015--1089, 2002.

\bibitem{S81}
P.~Sarnak.
\newblock Asymptotic behavior of periodic orbits of the horocycle flow and
  eisenstein series.
\newblock {\em Communications on Pure and Applied Mathematics}, 34:719--739,
  1981.

\bibitem{SU}
J.~Smillie and C.~Ulcigrai.
\newblock Beyond {S}turmian sequences: coding linear trajectories in the
  regular octagon.
\newblock {\em Proc. Lond. Math. Soc. (3)}, 102(2):291--340, 2011.

\bibitem{V89}
W.~A. Veech.
\newblock Teichm{\"u}ller curves in moduli space, {E}isentein series and an
  application to triangular billiards.
\newblock {\em Inventiones mathematicae}, 97:553--583, 1989.

\bibitem{Vee98}
W.~A. Veech.
\newblock Siegel measures.
\newblock {\em Annals of Mathematics}, 148:865--944, 1998.

\bibitem{Vorobets}
Y.~B. Vorobets.
\newblock Plane structures and billiards in rational polygons: the {V}eech
  alternative.
\newblock {\em Uspekhi Mat. Nauk}, 51(5(311)):3--42, 1996.

\bibitem{Zorich}
A.~Zorich.
\newblock Flat surfaces.
\newblock {\em Frontiers in number theory, physics, and geometry I}, pages
  439--585, 2006.

\end{thebibliography}
\end{document}